\documentclass[12pt]{amsart}
\usepackage{amssymb}

\usepackage[T1]{fontenc}
\usepackage{xcolor}

\usepackage{hyperref}

\newcommand{\bbN}{\mathbb{N}}
\newcommand{\bbQ}{\mathbb{Q}}
\newcommand{\bbR}{\mathbb{R}}
\newcommand{\bbZ}{\mathbb{Z}}

\DeclareMathOperator{\card}{Card}
\DeclareMathOperator{\Imag}{Im}
\DeclareMathOperator{\rad}{rad}

\newtheorem{thm}{Theorem}

\newtheorem{lem}[thm]{Lemma}

\theoremstyle{definition}
\newtheorem*{rem-nonum}{Remark}
\newtheorem*{remarks-nonum}{Remarks}

\begin{document}

\title{Perfect powers in sequences of polygonal numbers}

\author[A. D\k{a}browski, S. E. Rihane, G. Soydan and P. M. Voutier]{Andrzej D\k{a}browski, Salah Eddine Rihane, G\"{o}khan Soydan and Paul M. Voutier}

\date{}

\begin{abstract} 
Let $P_s(n)$ denote the $n$-th $s$-gonal number. Consider the Diophantine equation $P_{s}(n) = t^{m}$ for integers $n, s, t$ and $m > 2$. 
All solutions to this equation are known for $m>2$ and $s\in\{3,5,6,8,10,20\}$. Here we extend these results 
to the cases $s = 2k+4$ (where $k = 4,6$ or $5 \leq k \leq 97$ is a prime number) and $s = k+4$ (where $k = 9,15$ 
or $3 \leq k \leq 97$ is a prime number).  
The proofs of our results use the modular and hypergeometric methods, linear forms in logarithms and extensive calculations. 
We were unable to completely solve the above Diophantine equations, but we expect (based on GRH and the weak effective $abc$ conjecture) 
that there will be no additional solutions beyond those explicitly shown in Theorems~\ref{thm: s even},
\ref{thm: s odd} and \ref{thm: s remaining}. 
\end{abstract} 



\maketitle

\section{Introduction} 

The polygonal numbers (sometimes called figurate numbers) have a long history in number theory, going back to at least Fermat. They even appear in Gauss's diary, where he famously wrote ``$\text{EYPHKA! num} = \triangle + \triangle + \triangle$'' in 1796.
This result, as well as the four square theorem proved by Lagrange, is a special case of Fermat's polygonal number theorem which states that for $s \geq 3$, every positive integer can be written as a sum of at most $s$ $s$-gonal numbers.
For more information about the history and properties of polygonal numbers, we recommend \cite{DD, D}. Numerous papers are dedicated to solving various Diophantine equations related to polygonal numbers. See \cite{GDP, HPTV, HPTV-2, KT, KPP, KR, M-R} for more information.

If $s$ is the number of sides in a regular polygon, the formula for the $n$-th $s$-gonal
number is
\[
P_{s}(n) = \frac{(s-2)n^{2} - (s-4)n}{2}. 
\]

Consider the Diophantine equation 
\begin{equation}
\label{eq:polygonal}
P_{s}(n) = t^{m}
\end{equation} 
for integers $n, s, t$ and $m > 2$. All solutions to this equation are known for $m>2$ and $s\in\{3,5,6,8,20\}$ (see \cite{KPP}) 
and $s=10$ (see \cite{M-R}).

Here we extend these results. In fact, we will consider a very modest extension of them too.
We will consider equation~\eqref{eq:polygonal} where $n$ can also be negative. 
We may assume that $m=p$ is an odd prime or $m=4$. 

\begin{thm}
\label{thm: s even}
Let $5 \leq k \leq 97$ be a prime number. 

\noindent
{\rm (i)} The Diophantine equation $P_{2k+4}(n) = t^4$ has only the trivial solutions
$(n,t) = (0,0), (1,\pm 1)$ except for $k = 5, 17, 29$, where we have the additional
solutions $(n,t) = (-2000,\pm 70)$, $(n,t) = (-272,\pm 34)$ and $(n,t) = (81,\pm 21)$, respectively. 

\noindent
{\rm (ii)} The Diophantine equation $P_{2k+4}(n) = t^3$ has only the trivial solutions
$(n,t) = (0,0), (1,1), (k,k)$ except for the values
$k=5$ $($with the solutions $(n,t) = (-16974593,120019))$ and
$k=13$ $($with the solutions $(n,t) = (-1,3), (-8,10))$.

\noindent
{\rm (iii)} The Diophantine equation $P_{2k+4}(n) = t^p$ $($with $p \geq 5$ a prime$)$
has the solutions $(p,n,t) = (p,0,0), (p,1,1)$, and there are no further solutions
with $p = 5, 7$ or $p > p_0(2k+4)$, where the values $p_0(2k+4) := P_{k,10^{100,000}}$
are listed in Table~$\ref{table:thm-2}$.

\noindent
{\rm (iv)} For each prime $11 \leq p \leq p_0(2k+4)$, there is at most one other solution
to the Diophantine equation $P_{2k+4}(n) = t^p$ 
and any such solution satisfies $n \geq 10^{100,000p}$. 

\noindent
{\rm (v)} If we assume the Generalised Riemann Hypothesis, then the Diophantine equation
$P_{2k+4}(n) = t^p$ has no further solutions with $11 \leq p \leq 31$. 

\noindent
{\rm (vi)} The weak effective $abc$ conjecture $abc(r)$ $($with $r=1.63)$ implies that
$(n,t)=(1,1)$ is the only solution of the Diophantine equation  
$P_{2k+4}(n) = t^p$ provided $p > 31$. 
\end{thm} 

\begin{thm} \label{thm: s odd} 
Let $3 \leq k \leq 97$ be a prime number. 

\noindent
{\rm (i)} The Diophantine equation $P_{k+4}(n) = t^4$ has only the trivial solutions
$(n,t) = (0,0), (1,\pm 1)$ except for $k =3, 17$, where we have the additional
solutions $(n,t) = (6, \pm 3)$ and $(n,t) = (-32,\pm 10)$, respectively.

\noindent
{\rm (ii)} The Diophantine equation $P_{k+4}(n) = t^3$ has only the trivial solutions
$(n,t) = (0,0), (1,1)$ except for the values
$k=3$ $($with the solutions $(n,t) = (-3,3), (-9,6), (-48,18))$,
$k=7$ $($with the solutions $(n,t) = (-1,2), (-27,15))$,
$k=23$ $($with the solutions $(n,t) = (2,3), (66125,3795))$,
$k=31$ $($with the solution $(n,t) = (-961,248))$,
$k=41$ $($with the solution $(n,t) = (-16,18))$,
$k=83$ $($with the solution $(n,t) = (10375,1660))$.

\noindent
{\rm (iii)} The Diophantine equation $P_{k+4}(n) = t^p$ $($with $p \geq 5$ a prime$)$
has the solutions $(p,n,t) = (p,0,0), (p,1,1)$, and there are no further solutions
with $p = 5, 7$ or $p > p_0(k+4)$, where the values $p_0(k+4) := P_{k,10^{100,000}}$
are listed in Table~$\ref{table:thm-2-2}$.

\noindent
{\rm (iv)} For each prime $11 \leq p \leq p_0(k+4)$, there is at most one other solution
to the Diophantine equation $P_{k+4}(n) = t^p$ and any such solution satisfies
$n \geq 10^{100,000p}$.

\noindent
{\rm (v)} If we assume the Generalised Riemann Hypothesis, then the Diophantine equation
$P_{k+4}(n) = t^p$ has no further solutions with $11 \leq p \leq 29$. 

\noindent
{\rm (vi)} The weak effective $abc$ conjecture $abc(r)$ $($with $r=1.63)$ implies that
$(n,t)=(1,1)$ is the only solution of the equation $P_{k+4}(n) = t^p$ provided
$p > 29$.
\end{thm}

Let us include the results for the Diophantine equation $P_s(n) = t^m$, with
$s \in \{12, 13, 16, 19\}$ 
(to collect results, complete or not, for all $s$ up to $20$). 

\begin{thm} \label{thm: s remaining} 
\noindent
{\rm (i)} The Diophantine equation $P_s(n) = t^4$ $($with $s = 12, 13, 16, 19)$ has the trivial
solutions $(n,t) = (0,0), (1,\pm 1)$, and additional solutions $(n,t)=(-64,\pm 12)$
for $s=12$,
additional solutions $(n,t)=(2,\pm 2)$, $(-3,\pm 3)$, $(-54,\pm 12)$ for $s=16$, 
and additional solutions $(n,t)=(-1,\pm 2)$ for $s=19$. 

\noindent
{\rm (ii)} The Diophantine equation $P_s(n) = t^3$ $($with $s = 12, 13, 16, 19)$ has
the trivial solutions $(n,t) = (0,0), (1,1)$, and additional solutions $(n,t)=(4,4)$,
$(5488,532)$ for $s=12$, and additional solution $(n,t)=(6,6)$ for $s=16$. 

\noindent
{\rm (iii)} The Diophantine equation $P_s(n) = t^p$ $($with $s = 12, 13, 16, 19)$ has
the solutions $(p,n,t) = (p,0,0), (p,1,1)$, and there are no further solutions
with $p = 5, 7$ or $p > p_0(s)$, where $p_0(12) = 251$, $p_0(13) = 281$, 
$p_0(16) = 307$, and $p_0(19) = 346$. 

\noindent
{\rm (iv)} For each prime $11 \leq p \leq p_0(s)$, there is at most one other solution
to the Diophantine equation $P_{s}(n) = t^p$ $($with $s = 12, 13, 16, 19)$ and
any such solution satisfies $n \geq 10^{100,000p}$.

\noindent
{\rm (v)} If we assume the Generalised Riemann Hypothesis, then the Diophantine equation
$P_{s}(n) = t^p$ $($with $s = 12, 13, 16, 19)$ has no further solutions with
$11 \leq p \leq 31$. 

\noindent
{\rm (vi)} The weak effective $abc$ conjecture $abc(r)$ $($with $r=1.63)$ implies that
$(n,t)=(1,1)$ is the only solution of the equation  
$P_s(n) = t^p$ $($with $s = 12, 13, 16, 19)$ provided $p > 31$.
\end{thm}

\begin{rem-nonum}
(i) We can improve the upper bounds $p_0(s)$ on $p$ in parts~(iii)
of the above Theorems. 
In Appendix~\ref{sect:A-1} we give details for $s = 12, 14, 16, 18$.

(ii) Our calculations in Sections~\ref{calculations for s even}, \ref{calculations for s odd} and \ref{sect:cases s = 16, 19}  frequently 
use $\mathtt{PARI/GP}$. 
Our use of the \verb+thueinit()+ function itself is independent of GRH if we knew that the class number of the related number field were $1$ (or if the constant term in the Thue equation were $\pm 1$). But it turns out that we only know that the class number is $1$ 
conditionally, depending on GRH. The obvious thing is to run \verb+bnfcertify()+ on all the results that we have from the 
\verb+thueinit()+ function. But even for ``low degree'' cases, \verb+bnfcertify()+ would take far too long. So all the Thue equations 
that we solved using $\mathtt{PARI/GP}$ where the constant term in the Thue equation is not $\pm 1$ are currently dependent on GRH.  
\end{rem-nonum}

The proofs of our results use the modular and hypergeometric methods, linear forms
in logarithms and extensive computations in $\mathtt{MAPLE}$, $\mathtt{MAGMA}$, and $\mathtt{PARI/GP}$. 
We expect (based on GRH and the weak effective $abc$ conjecture) that there will be no additional solutions beyond those explicitly 
shown in Theorems~\ref{thm: s even}, \ref{thm: s odd} and \ref{thm: s remaining}. 

Part~(i) in each of Theorems~\ref{thm: s even}-\ref{thm: s remaining} follows
immediately from the results in Section~\ref{sect:m=4}.
Part~(ii) in each of these theorems follows immediately from the results in
Sections~\ref{calculations for s even}, \ref{calculations for s odd} 
and \ref{sect:cases s = 16, 19}. 
Parts~(iii) and (iv) in each of these theorems follow immediately from the results
in Sections~\ref{sect:p-power-even} and \ref{sect:p-power-odd}.
Part~(v) in each of Theorems~\ref{thm: s even}-\ref{thm: s remaining} follows
immediately from the results in Sections~\ref{calculations for s even}, \ref{calculations for s odd} 
and \ref{sect:cases s = 16, 19}.  
Part~(vi) of each of them follows immediately from Lemmas~\ref{lem:effective abc},
\ref{lem:effective abc 2}, \ref{lem: even s Case III} and \ref{lem: odd s Case IIIb}.

Programs and numerical data used in our article are available at: 

\href{https://github.com/gsoydan74/Polygonal}{https://github.com/gsoydan74/polygonal} 

{\bf Acknowledgements.}
The numerical calculations reported in this paper were partially performed at
T\"{U}B\.{I}TAK (The Scientific and Technological Research Council of T\"{u}rkiye)
ULAKB\.{I}M, High Performance and Grid Computing Center (ARF and TRUBA resources).
G. S. cordially thanks Sevil Sar{\i}kurt Malcio\u{g}lu and \.{I}smail G\"uzel
who are senior researchers in T\"UB\.ITAK-ULAKB\.IM for their kind and patient
assistance with the use of ARF-TRUBA resources and thanks Adnan K{\i}l{\i}\c{c} (who
is in the physics department of Bursa Uluda\u{g} University) for his help and
for introducing ARF-TRUBA resources to him. G. S. would also like to give special
thanks to Bill Allombert, who is the main developer of the $\mathtt{PARI/GP}$ computer
algebra system, for his generous help and suggestions as well as for preparing a
snapshot of the new version of $\mathtt{PARI/GP}$ ($\mathtt{PARI/GP}$ version 2.16,
git branch \textit{bill-parbnf-gokhan}) so that some equations in the paper can be
solved with PARI/GP. Two computations in this paper were carried out using the PlaFRIM
experimental testbed, supported by Inria, CNRS (LABRI and IMB), Universit\'{e} de
Bordeaux, Bordeaux INP and Conseil R\'{e}gional d'Aquitaine (see https://www.plafrim.fr).
G. S. was supported by the Research Fund of Bursa Uluda\u{g} University under Project No: FGA-2023/1545.
This paper began when the second author was visiting the third author. The second author
expresses gratitude to Bursa Uluda\u{g} University for offering an excellent working
environment and for its hospitality. The first author would like to thank very much
the third author for visiting him in Szczecin, and for very fruitful discussions.
We also thank Tomasz J\k{e}drzejak for preparing $\mathtt{MAGMA}$ code to use the
Sophie Germain type argument. 

\section{Fourth powers in sequences of polygonal numbers}
\label{sect:m=4}

For $m=4$, equation~\eqref{eq:polygonal} can be written as
\[
x^{2}=8(s-2)t^{4}+(s-4)^{2},
\]
with $x=2(s-2)n-(s-4)$. To find the integer solutions $(x,t)$ of the above
equation, we use the subroutine $\mathtt{IntegralQuarticPoints}$ of $\mathtt{MAGMA}$.
For $s=12,13,19$ or $s=k+4$ with $3\leq k\leq 97$ or $s=2k+4$ with $5\leq k\leq 97$
being a prime number, all the integer solutions are given by $(x,t)=(\pm (s-4),0)$
and $(\pm s, \pm 1)$, except that for $s=7,12,14,19,21,23,38,57$ and $62$ we
have the additional integer solutions $(\pm 57,\pm 3)$, $(\pm 1288,\pm 12)$,
$(\pm 48010, \pm 70)$, $(\pm 49, \pm 2),$ $(\pm 1233, \pm 10)$, $(\pm 467,\pm 6)$,
$(\pm 19618, \pm 34)$, $(\pm 757, \pm 6)$, and $(\pm 9662,\pm 21)$ respectively.
Furthermore, for $s=16$, all the integer solutions are given by $(x,t)=(\pm 12,0)$
and $(\pm 16, \pm 1)$, $(\pm 44, \pm 2)$, $(\pm 96, \pm 3)$, $(\pm 1524, \pm 12)$.
This leads to the following results.

\begin{lem}
\noindent
{\rm (a)}
Let $5\leq k\leq 97$ be a prime number. The Diophantine equation $P_{2k+4}(n)=t^{4}$
has only the trivial solutions $(n,t)=(0,0),(1,\pm 1)$ except for $k=5$, $17$ and
$29$ (that is, $s=2k+4=14$, $38$ and $62$), where we have the additional solutions
$(-2000,\pm 70)$, $(-272,\pm 34)$ and $(81, \pm 21)$, respectively.

\noindent
{\rm (b)}
The Diophantine equation $P_{12}(n)=t^{4}$ has only the trivial solutions
$(n,t)=(0,0)$, $(1,\pm 1)$, $(-64,\pm 12)$.

\noindent
{\rm (c)}
The Diophantine equation $P_{16}(n)=t^{4}$ has only the trivial solutions
$(n,t)=(0,0),(1,\pm 1),(2,\pm 2),(-3,\pm 3),(-54,\pm 12)$.
\end{lem}  

\begin{lem}
\noindent
{\rm (a)}
Let $3 \leq k \leq 97$ be a prime number. The Diophantine equation $P_{k+4}(n)=t^{4}$
has only the trivial solutions $(n,t)=(0,0)$, $(1,\pm 1)$ except for $k=3,17$ (that
is, $s=k+4=7,21$), where we have the additional solutions $(6, \pm 3)$ and $(-32,\pm 10)$,
respectively.

\noindent
{\rm (b)}
The Diophantine equation $P_{13}(n)=t^{4}$ has only the trivial solutions $(n,t)=(0,0),(1,\pm 1)$.

\noindent
{\rm (c)}
The Diophantine equation $P_{19}(n)=t^{4}$ has only the trivial solutions $(n,t)=(0,0)$,$(1,\pm 1)$, $(-1,\pm 2)$.
\end{lem}

\section{The Diophantine equations $Ax^{p} - By^{p} = A - B$, with $A > B > 0$} \label{sect:p-power}

Given positive integers $A > B$, consider the following Diophantine equation 
\begin{equation} \label{eq:Thue general} 
Ax^{p} - By^{p} = A - B. 
\end{equation} 

\begin{remarks-nonum}
(i) The famous $abc$ conjecture implies that there exists a constant
$g(A,B)$ such that for any prime $p > g(A,B)$ the equation~\eqref{eq:Thue general}
only has the solution $(x,y) = (1,1)$ (see, for instance, \cite[Conjecture 2]{IK}). 

(ii) The weak effective $abc$ conjecture implies that $(x,y)=(1,1)$ is the only
solution of the equation~\eqref{eq:Thue general} for $p \geq g_{0}(A,B)$, with
$g_{0}(A,B)$ an effective constant. Lemma~\ref{lem:effective abc} below gives
explicit bounds for some pairs of $(A,B)$.

(iii) We can use \cite[Theorem~3]{Evertse} to prove that the equation~\eqref{eq:Thue general}
has at most one solution 
different from $(x,y)=(1,1)$ for $p\geq h(A,B)$, with $h(A,B)$ an effective constant.
Lemma~\ref{lem:Evertse1} below gives explicit bounds for some pairs of $(A,B)$
applicable to our work here, in particular, for equation~\eqref{eq:Thue1}. 
\end{remarks-nonum}

In the next two sections, we consider the family~\eqref{eq:Thue general} for
$(A,B) = (k+1,1)$ and $(k+2,2)$. 

\section{The Diophantine equations $(k+1)x^{p} - y^{p} = k$} 
\label{sect:p-power-even}

Given a positive integer $k$, consider the following Diophantine equation 
\begin{equation} \label{eq:Thue1} 
(k+1)x^{p} - y^{p} = k. 
\end{equation} 

Bennett \cite{Ben} showed by means of the hypergeometric method that the equation
$\left| (a+1)x^{n} - ay^{n} \right| = 1$ ($a = 1, 2, \ldots$) has no solution with
$|xy| > 1$. This implies that for $k=1$, equation~\eqref{eq:Thue1} has no solution with $|xy| > 1$. 

The equation~\eqref{eq:Thue1} for $k = 2$, $3$, $8$ also has no solution with $|xy| > 1$.
This follows from \cite[Theorem~1.5]{BVY} for $k = 2,8$, 
and from the recipes of \cite{BS} (see \cite{M-R}) for $k = 3$.

The above results were important steps in solving the equation~\eqref{eq:polygonal} for $s = 6, 8, 10, 20$. Let us stress that the 
above four cases are the only ones where the modular method allows the complete solution
of the equation~\eqref{eq:Thue1} by associating it with
generalised Fermat equations of signature $(p,p,2)$, $(p,p,3)$ and $(p,p,p)$ (see \cite[Remark 3.1]{M-R}). 

Below we will consider the equation~\eqref{eq:Thue1} for  primes $5\leq k \leq 97$ and $k=4$ or $6$.

\subsection{Number of Solutions of \eqref{eq:Thue1}}

We can use \cite[Theorem~3]{Evertse} to obtain the following.

\begin{lem} \label{lem:Evertse1}
The equation~\eqref{eq:Thue1} has at most one solution different from $(x,y)=(1,1)$ for $p \geq \max \{ 17, 2.1\log(k)+7.5 \}$.
\end{lem}

\begin{proof}
Assume that $x,y>0$. In the notation of \cite[Theorem~3]{Evertse}, here we have $a=k+1$, $b=1$, $C=k$ and $n=p \geq 17$. So $M_{n}=8.4n=8.4p$ and $A_{n}=2.4$. Hence there is at most one solution of
$0 < \left| (k+1)x^{p}-y^{p} \right| \leq k$ with
$x, y \in \bbN=\{ 1,2,\ldots \}$, $\gcd(x,y)=1$ and
$\max \left\{ (k+1)x^{p}, y^{p} \right\} \geq M_{n}C^{A_{n}}=8.4pk^{2.4}$ (since $b=1$ and $0<y^{p}<(k+1)x^{p}$, $\left| by^{p} \right|$ in \cite[Theorem~3]{Evertse} is $y^{p}$ here).

The only solution of \eqref{eq:Thue1} with $x=1$ is $(x,y)=(1,1)$, so we may assume that $x \geq 2$. Thus,
$\max \left\{ (k+1)x^{p}, \left| by^{p} \right| \right\}>k2^{p}$ and it remains to show that $k2^{p} \geq 8.4pk^{2.4}$ holds here. This will hold if $p>3.1+\log(p)/\log(2)+2.0198\log(k)$. Using the fact that for $z \geq 1$, we have
\[
0.02641z+2.85>\log(z)
\]
we find that for $p \geq 38$, we have $\left( 0.02641/\log(2) \right)p+2.85/\log(2)>\log(p)/\log(2)$.
If $p \geq 2.1\log(k)+7.5$, then
$\left( 1 - 0.02641/\log(2) \right)p -2.85/\log(2)>3.1+2.1\log(k)$. Combining these two inequalities, the desired inequality, $p>3.1+\log(p)/\log(2)+2.0198\log(k)$
holds for $p \geq 2.1\log(k)+7.5$. 
We have $3.1 + \log(17)/\log(2) + 2.0198\log(97) = 16.427464\ldots$,
hence $p >  3.1 + \log(p)/\log(2) + 2.0198\log(k)$ holds for any integer
$4 \leq k \leq 97$ and prime $p \geq 17$. 
If $x,y<0$, the equation~\eqref{eq:Thue1} can be transformed into $(-y)^p-(k+1)(-x)^p=k$. Furthermore, switching between $ax^n$ and $by^n$ in \cite[Theorem~3]{Evertse} has no impact on the final result.
\end{proof}

A special case of the weak effective $abc$ conjecture ($abc(r)$) \cite[page~104]{Browkin}
states that if $r \geq 1.63$ is a fixed real number, then there are no relatively
prime triples of positive integers, $(a,b,c)$, with $a+b=c$ and $L(a,b,c)=\log(c)/\log \rad (abc)>r$.

We can apply this conjecture to our work here.

\begin{lem}
\label{lem:effective abc} 
The weak effective $abc$ conjecture ($abc(r)$) implies that $(x,y)=(1,1)$ is the only solution of the equation~\eqref{eq:Thue1} provided $p>\max \{ 17, 2r+3r\log(k+1) \}$ for some $r \geq 1.63$.

In particular, with $r=1.63$, the weak effective $abc$ conjecture ($abc(r)$) implies that $(x,y)=(1,1)$ is the only solution of the equation~\eqref{eq:Thue1}
when $p>23$ for all primes $3\leq k\leq 97$ or $k = 4,6$.
\end{lem} 

\begin{proof}
If $x,y>0$, we apply the weak effective $abc$ conjecture with $a=y^{p}$, $b=k$
and $c=(k+1)x^{p}$ (we will specify $r$ shortly). We have $0<y^{p}<(k+1)x^{p}$,
so $L(a,b,c) \geq p\log(x)/ \log(k(k+1)xy)>p\log(x)/ \log \left( (k+1)^{2+1/p}x^{2} \right)$.
For fixed $k$, this is a monotonically increasing function of $x$. If there is a
solution of \eqref{eq:Thue1} with $(x,y) \neq (1,1)$, then $x \geq 2$, so
$L(a,b,c)>p\log(x)/ \left( 2.06\log(k+1)+2\log(x) \right)$, since we may assume
that $p \geq 17$. Hence $L(a,b,c)>r$, if $p>2r+3r\log(k+1)$.

If $x,y<0$, we can rewrite \eqref{eq:Thue1} as 
\begin{equation}\label{eq:Thue1-2} 
(k+1)X^{p} - Y^{p} = -k,
\end{equation} 
with $X,Y>0$. We apply the weak effective $abc$ conjecture with $a=(k+1)X^{p}$,
$b=k$ and $c=Y^{p}$. We have $0<X^{p}<Y^{p}$, so
$L(a,b,c) \geq p\log(Y)/ \log(k(k+1)XY)>p\log(Y)/ \log \left( (k+1)^{2}Y^{2} \right)$.
For fixed $k$, this is a monotonically increasing function of $Y$. Since
$(X,Y) = (1,1)$ is not a solution of \eqref{eq:Thue1-2}, we have $Y \geq 3$.
Hence $L(a,b,c)>r$, if $p>2r+1.8r\log(k+1)$.
\end{proof}

We now obtain upper bounds for the exponent $p$ such that the Thue equations
in \eqref{eq:Thue1} have solutions with $|x|>1$. We proceed in an iterative way using
lower bounds for linear forms in the logarithms of two algebraic numbers due to
Laurent \cite{Laurent}, as well as continued fraction calculations.

\subsection{Linear forms in logarithms}
\label{subsect:lfl-1}

We first provide the linear form that arises here.

We can rewrite the equation~\eqref{eq:Thue1} as
\begin{equation}
\label{eq:lemL-1}
k+1 - \frac{k}{x^{p}} = \left( \frac{y}{x} \right)^{p}.
\end{equation}

So we see that $y/x>1$.
Suppose that $x,y>0$. The case of $x,y<0$ follows in the very same way by using
$-x$ and $-y$ instead of $x$ and $y$.
If $x=1$, then from equation~\eqref{eq:Thue1},
we see that $y=1$ too. So we may assume that $y>x \geq 2$.

From \eqref{eq:lemL-1} and $y>2$, we can use Exercise~1.1(b) of \cite{Wald}
with $z=(k+1)(x/y)^{p}$ and $\theta=k/y^{p}<1/2$ (from $p \geq 150$ and $k<3^{149}$)
to obtain
\begin{equation}
\label{eq:lfl}
\left\vert \Lambda \right\vert \leq \frac{2k}{y^{p}},\quad \text{where}\quad
\Lambda:=p\log \left( \frac{y}{x} \right) - \log (k+1).
\end{equation}

We initially apply \cite[Theorem~2]{Laurent}. We start by presenting some notation
from his paper.

Let $\alpha_{1}$ and $\alpha_{2}$ be
real algebraic numbers with $\left| \alpha_{1} \right| \geq 1$ and $\left| \alpha_{2} \right| \geq 1$.
We consider the linear form
\[
\Lambda=b_{2} \log \alpha_{2}-b_{1} \log \alpha_{1},
\]
where $\log \alpha_{1}$, $\log \alpha_{2}$ are any determinations of the logarithms
of $\alpha_{1}$, $\alpha_{2}$ respectively, and $b_{1}$, $b_{2}$ are positive
integers. Put
\[
D=\left[ \bbQ \left( \alpha_{1},\alpha_{2} \right) : \bbQ \right]/\left[ \bbR \left( \alpha_{1},\alpha_{2} \right):\bbR \right].
\]

Let $\alpha$ be any non-zero algebraic number with $a_{0} \prod_{j=1}^{d} \left( X-\alpha^{(j)} \right)$
as its minimal polynomial over $\bbZ$. We denote by
\[
h(\alpha)=\frac{1}{d} \left( \log \left| a_{0} \right| + \sum_{j=1}^{d} \log \max \left\{1, \left| \alpha^{(j)} \right| \right\} \right),
\]
the absolute logarithmic height of $\alpha$.

For $i=1,2$, we put
\[
a_{i} \geq \max \left\{ 1, \varrho \left| \log \alpha_{i} \right| - \log \left| \alpha_{i} \right| + 2Dh \left( \alpha_{i} \right) \right\}.
\]

We apply Laurent's Theorem~2 with $b_{1}=1$, $b_{2}=p$, $\alpha_{1}=k+1$,
and $\alpha_{2}=y/x$. For this choice, we have
$h \left( \alpha_{1} \right)=\log (k+1)$, $h \left( \alpha_{2} \right)=\log y$,
and $D=1$. So we can take
$a_{1}:=\rho \left| \log \alpha_{1} \right| - \log \left| \alpha_{1} \right|+2Dh\left( \alpha_{1} \right)=(\rho+1)\log (k+1)$.
From equation~\eqref{eq:lemL-1}, we have $y/x<(k+1)^{1/p}$, so we put
$a_{2}:=(\rho-1) \log(k+1)/p+2\log(y)
\geq \rho \left| \log \alpha_{2} \right| - \log \left| \alpha_{2} \right|+2Dh\left( \alpha_{2} \right)$.

For each value of $k$, we first assumed that $y \geq 3$ and used a program written
in Maple (the function \verb+check_1()+ in \verb+Laurent-Thm-2-estimate.txt+) to
search for values of $\varrho$ and $\mu$ that gave the best lower bound
for $\log | \Lambda |$ in Theorem~2 of \cite{Laurent}. Dividing these lower bounds
by $\log(3)$ ($3$ here is the lower bound for $y$) and comparing this expression
with the upper bound from \eqref{eq:lfl}, we obtained an upper bound for $p$. We
denote this upper bound by $P_{k,3}$.
The results are in the $y \geq 3$ column of Table~\ref{table:thm-2}.

By increasing the lower bound for $y$, we reduce the size of the lower order terms
in the lower bound for $\log |\Lambda|$ in Laurent's Theorem~2. So we conducted
a brute force search for solutions of \eqref{eq:Thue1} for each $k$ and for $p \leq P_{k,3}$
with $|y| \leq 1000$. No additional solutions were found. We then proceeded as in
the previous paragraph, but with the assumption that $y \geq 1001$. This gave us
an improved upper bound: $p \leq P_{k,1001}$.
The results are in the $y \geq 1001$ column of Table~\ref{table:thm-2}.

\begin{table}[ht!]
\centering
\begin{tabular}{|c|c|c|c|} \hline
$k$ & $\left( \varrho,\mu,P_{k,3} \right)$ & $\left( \varrho,\mu,P_{k,1001} \right)$ & $\left( \varrho,\mu,P_{k,10^{100,000}} \right)$ \\ \hline
$4$ & $(74/5, 1/3, 780)$ & $(181/10, 1/3, 402)$  & $(187/10, 1/3, 256)$ \\ \hline
$5$ & $(73/5, 1/3, 851)$ & $(179/10, 1/3, 440)$  & $(185/10, 1/3, 284)$ \\ \hline
$6$ & $(72/5, 1/3, 911)$ & $(177/10, 1/3, 472)$  & $(184/10, 1/3, 307)$ \\ \hline
$7$ & $(71/5, 1/3, 963)$ & $(176/10, 1/3, 499)$  & $(184/10, 1/3, 327)$ \\ \hline
$11$ & $(69/5, 1/3, 1121)$ & $(173/10, 1/3, 582)$  & $(91/5, 1/3, 388)$ \\ \hline
$13$ & $(137/10, 1/3, 1181)$ & $(86/5, 1/3, 614)$  &$(181/10, 1/3, 411)$\\ \hline
$17$ & $(27/2, 1/3, 1278)$ & $(171/10, 1/3, 665)$  &$(181/10, 1/3, 448)$ \\ \hline
$19$ & $(69/5, 1/3, 1319)$ & $(17, 1/3, 686)$  &$(18, 1/3, 464)$ \\ \hline
$23$ & $(133/10, 1/3, 1390)$ & $(169/10, 1/3, 723)$  & $(18, 1/3, 492)$ \\ \hline
$29$ & $(66/5, 1/3, 1477)$ & $(84/5, 1/3, 769)$  &$(18, 1/3, 525)$ \\ \hline
$31$ & $(131/10, 1/3, 1503)$ & $(84/5, 1/3, 782)$  &$(179/10, 1/3, 535)$ \\ \hline
$37$ & $(13, 1/3, 1570)$ & $(167/10, 1/3, 817)$  & $(179/10, 1/3, 560)$ \\ \hline
$41$ & $(129/10, 1/3, 1609)$ & $(167/10, 1/3, 837)$  &$(179/10, 1/3, 575)$ \\ \hline
$43$ & $(129/10, 1/3, 1628)$ & $(83/5, 1/3, 846)$  &$(179/10, 1/3, 582)$ \\ \hline
$47$ & $(129/10, 1/3, 1662)$ & $(83/5, 1/3, 864)$  &$(179/10, 1/3, 595)$ \\ \hline
$53$ & $(64/5, 1/3, 1708)$ & $(83/5, 1/3, 888)$  &$(179/10, 1/3, 613)$ \\ \hline
$59$ & $(127/10, 1/3, 1750)$ & $(33/2, 1/3, 909)$  &$(89/5, 1/3, 629)$ \\ \hline
$61$ & $(127/10, 1/3, 1763)$ & $(33/2, 1/3, 916)$  &$(89/5, 1/3, 634)$\\ \hline
$67$ & $(127/10, 1/3, 1799)$ & $(33/2, 1/3, 934)$  &$(89/5, 1/3, 648)$ \\ \hline
$71$ & $(63/5, 1/3, 1822)$ & $(82/5, 1/3, 946)$  &$(89/5, 1/3, 656)$ \\ \hline
$73$ & $(63/5, 1/3, 1833)$ & $(82/5, 1/3, 951)$  &$(89/5, 1/3, 660)$ \\ \hline
$79$ & $(63/5, 1/3, 1864)$ & $(82/5, 1/3, 967)$  &$(89/5, 1/3, 672)$ \\ \hline
$83$ & $(63/5, 1/3, 1883)$ & $(82/5, 1/3, 977)$  &$(89/5, 1/3, 679)$ \\ \hline
$89$ & $(25/2, 1/3, 1911)$ & $(82/5, 1/3, 991)$  & $(89/5, 1/3, 690)$ \\ \hline
$97$ & $(25/2, 1/3, 1945)$ & $(163/10, 1/3, 1008)$  & $(89/5, 1/3, 702)$ \\ \hline
\end{tabular}\caption{Results from Laurent's Theorem~2}\label{table:thm-2}
\end{table}

\subsection{Continued fractions}
\label{continued fractions 1}
By increasing the lower bound on $y$ even further, we can reduce the upper bound
on $p$ more. An efficient way to do this is using continued fractions. In particular,
the continued fraction expansion of the positive real root of $(k+1)X^{p}-1$,
which we label $\xi$. We start by obtaining a lower bound on $x$ such that $y/x$
is a convergent in the continued fraction expansion of $\xi$ for any solution
$(x,y)$ of the Thue equation~\eqref{eq:Thue1}.

\begin{lem}\label{lem:Y0}
If $(x,y)$ is a solution of \eqref{eq:Thue1} with $\gcd(x,y)=1$ and $|x| \geq 2$,
then $y/x$ is a convergent in the continued fraction expansion of $(k+1)^{1/p}$.
\end{lem}

\begin{proof} We use Lemma~1.1 of \cite{TdW}.
In the notation of Lemma~1.1 of \cite{TdW}, but swapping the roles of $x$ and $y$, we have $g(y)=y^{p}-(k+1)$, so $g'(y)=py^{p-1}$, $m=-k$, $n=p$, $s=1$, $t=(n-1)/2$, $\xi^{(1)}=\xi=(k+1)^{1/p}$ and $\xi^{i}=\xi \zeta_{p}^{i}$ for $i=s+1,\ldots,s+t$, where $\zeta_{p}=\exp(2 \pi i/p)$. Thus, $\min_{1 \leq i \leq t} \left| g' \left( \xi^{(s+i)} \right) \right| =p(k+1)^{(p-1)/p}$ and 
\[
\min_{1 \leq i \leq t} \left| \Imag \xi^{(s+i)} \right|
=(k+1)^{1/p}\sin(2\pi (p-1)/(2p))=(k+1)^{1/p}\sin(\pi/p),
\]
(the last equality holding since $\sin(\pi-\theta)=\sin(\theta)$).
Since $\sin(x)$ is concave on $[0,\pi/7]$, we have
$\sin(x) \geq x \sin(\pi/7)/(\pi/7)>0.966x$. So for $p \geq 7$, we have
$\sin(\pi/p)>3.03/p$.

Hence
\[
Y_{0} = \left\lceil \left( \frac{2^{p-1}kp}{p(k+1)^{(p-1)/p} \cdot 3.03(k+1)^{1/p}} \right)^{1/p} \right\rceil
< \left\lceil 2\left( \frac{k}{3.03(k+1)} \right)^{1/p} \right\rceil \leq 2.
\]
\end{proof}

\begin{lem}
\label{lem:cf-LB}
For each prime $5 \leq k \leq 97$ or $k=4,6$ and each prime, $p$, satisfying $7 \leq p \leq P_{k,1001}$,
there are no solutions of \eqref{eq:Thue1} with $|x|<10^{100,000}$.
\end{lem}

\begin{proof}
Using the \verb+contfrac()+ function in $\mathtt{PARI/GP}$ \cite{Pari}, for each prime $5 \leq k \leq 97$ 
or $k = 4,6$ 
and each $17 \leq p \leq P_{k,1001}$, we
calculated the partial quotients, $a_{n}$, and estimated from below the convergents,
$p_{n}/q_{n}$, in
the continued fraction expansion of $(k+1)^{1/p}$ until the denominator of the convergents
exceeded $10^{100,000}$. We determined the maximum of these partial quotients,
which we denote by $A$. It is well-known that
\[
\frac{1}{\left( a_{n+1}+2 \right) q_{n}^{2}}
< \left| (k+1)^{1/p}-\frac{p_{n}}{q_{n}} \right|.
\]

Hence, for all convergents, $p_{n}/q_{n}$, to $(k+1)^{1/p}$ with $q_{n} \leq 10^{100,000}$,
we have
\[
\frac{1}{\left( A+2 \right) q_{n}^{2}}
< \left| (k+1)^{1/p}-\frac{p_{n}}{q_{n}} \right|.
\]

The first inequality in Lemma~1.1(i) of \cite{TdW} also provides an upper bound
\[
\left| (k+1)^{1/p}- p_{n}/q_{n} \right| < \frac{2^{p-1}k^{1/p}}{p} \left| q_{n} \right|^{-p},
\]
when $\left( q_{n}, p_{n} \right)$ is a solution of equation~\eqref{eq:Thue1}.

Combining these upper and lower bounds, we obtain an upper bound for $q_{n}$
such that $\left( q_{n}, p_{n} \right)$ is a solution of equation~\eqref{eq:Thue1}.
This bound is small enough (typically $x \leq 2$)
that a brute force search suffices.

For example, for $k=4$ and $p=257$, we found that $q_{n}>10^{100,000}$ for $n=194,526$
and that the maximum partial quotient, $530,085$, occurred for $n=41,396$.
With $C_{1}=3.63 \cdot 10^{74}$ and $Y_{1}=2$ in Lemma~1.1 of \cite{TdW}, we found
that equation~\eqref{eq:Thue1} can only hold if $q_{n} \leq 2$.
This calculation, using 220,000 digits of precision, took under 4.3 seconds.

The above calculations were done using the \verb+check_all()+ and \verb+check1()+
functions we wrote, which are contained in \verb+cf-checks.gp+. 
\end{proof}

Applying Laurent's theorem again (using our Maple function \verb+check_1()+ in
\verb+Laurent-Thm-2-estimate.txt+), we obtain the bounds $p \leq P_{k,10^{100,000}}$
in Table~\ref{table:thm-2}.

\section{The Diophantine equations $(k+2)x^{p} - 2y^{p} = k$} 
\label{sect:p-power-odd}

Given a positive integer $k$, consider the following Diophantine equation
\begin{equation}\label{eq:Thue2} 
(k+2)x^{p} - 2y^{p} = k. 
\end{equation}
The equation~\eqref{eq:Thue2} has no solution with $|xy| > 1$ for $k=1$. This
follows for instance from \cite[Theorem~1.1]{BMS}. 
Note that below we use $k:=s-4$ for odd $s \geq 7$.

\subsection{Number of Solutions of \eqref{eq:Thue2}}

We can use \cite[Theorem~3]{Evertse} to obtain the following.

\begin{lem}
\label{lem:Evertse2}
The equation~\eqref{eq:Thue2} has at most one solution different from $(x,y)=(1,1)$
for $p \geq \max \{ 17, 2.1\log(k)+7.5 \}$.
\end{lem}

\begin{proof}
We use the same steps as in the proof of Lemma~\ref{lem:Evertse1}.
\end{proof}

Applying the weak effective $abc$ conjecture and following a similar argument as in the proof of Lemma~\ref{lem:effective abc}, we obtain:

\begin{lem}\label{lem:effective abc 2} 
The weak effective $abc$ conjecture ($abc(r)$) implies that $(x,y)=(1,1)$ is the
only solution of the equation~\eqref{eq:Thue2} provided $p>\max \{ 17, 3r+3r\log(k+2) \}$
for some $r\geq 1.63$.

In particular, with $r=1.63$, the weak effective $abc$ conjecture ($abc(r)$) implies that $(x,y)=(1,1)$ is the only solution of the equation~\eqref{eq:Thue2} for $k = 3,5,7,9$ and $p\geq 17$, $k=11,13,15$ and $p\geq 19$. 
\end{lem}

\subsection{Linear forms in logarithms}

In order to obtain an upper bound for the exponent $p$ such that the Thue equations in \eqref{eq:Thue2} have solutions with $|x|>1$, 
we will apply lower bounds for linear forms in the logarithms of two algebraic numbers. 
We can rewrite the equation~\eqref{eq:Thue2} as
\begin{equation}\label{eq:lemL-2}
k+2 - \frac{k}{x^{p}} = 2\left( \frac{y}{x} \right)^{p}.
\end{equation}

As in Subsection~\ref{subsect:lfl-1}, we see that $y/x>1$ and suppose that $x,y>0$.
If $x=1$, then from equation~\eqref{eq:Thue2}, we see that $y=1$ too. So we may
assume that $y>x \geq 2$. From \eqref{eq:lemL-2} and $y>2$, we can use Exercise~1.1(b) of \cite{Wald} with $z=(k+2)/2(x/y)^{p}$ and $\theta=k/y^{p}<1/2$ (from $p \geq 150$ and $k<3^{149}$) to obtain
\begin{equation}\label{eq:lfl-2}
\left\vert \Lambda \right\vert \leq \frac{k}{y^{p}},\quad \text{where}\quad
\Lambda:=p\log \left( \frac{y}{x} \right) - \log \left( \frac{k+2}{2} \right).
\end{equation}
We apply Laurent's Theorem~2 with $b_{1}=1$, $b_{2}=p$, $\alpha_{1}=(k+2)/2$, and $\alpha_{2}=y/x$. For this choice, we have
$h \left( \alpha_{1} \right)=\log (k+2)$, $h \left( \alpha_{2} \right)=\log y$,
and $D=1$. So we can take $a_{1}:=\rho \left| \log \alpha_{1} \right| - \log \left| \alpha_{1} \right|+2Dh\left( \alpha_{1} \right)=(\rho+1)\log \left(\frac{k+2}{2}\right) +2\log 2$. From equation~\eqref{eq:lemL-2}, we have $y/x<((k+2)/2)^{1/p}$, so we put $a_{2}:=(\rho-1)/p \log((k+2)/2)+2\log(y)
\geq \rho \left| \log \alpha_{2} \right| - \log \left| \alpha_{2} \right|+2Dh\left( \alpha_{2} \right)$.

Using again a program written in Maple (the function \verb+check_1()+ in \verb+Laurent-Thm-2-estimate.txt+), 
we collect all the information in Table~\ref{table:thm-2-2}.

\begin{table}[ht!]
\centering
\begin{tabular}{|c|c|c|c|} \hline
$k$ & $\left( \varrho,\mu,P_{k,3} \right)$ & $\left( \varrho,\mu,P_{k,1001} \right)$ & $\left( \varrho,\mu,P_{k,10^{100,000}} \right)$ \\ \hline
$3$ & $(18, 1/3,536)$ & $(43/2, 1/3, 268)$  & $(22, 1/3, 162)$ \\ \hline
$5$ & $(167/10, 1/3, 671)$ & $(101/5, 1/3, 341)$ & $(104/5, 1/3, 213)$ \\ \hline
$7$ & $(161/10, 1/3, 771)$ & $(98/5, 1/3, 394)$  & $(101/5, 1/3, 251)$ \\ \hline
$9$ & $(157/10, 1/3, 850)$ & $(96/5, 1/3, 436)$  & $(199/10, 1/3, 281)$ \\ \hline
$11$ & $(77/5, 1/3, 915)$ & $(189/10, 1/3,471)$  & $(197/10, 1/3, 306)$ \\ \hline
$13$ & $(151/10, 1/3, 971)$ & $(187/10, 1/3, 501)$  & $(39/2, 1/3, 328)$ \\ \hline
$15$ & $(149/10, 1/3, 1020)$ & $(37/2, 1/3, 526)$  & $(97/5, 1/3, 346)$ \\ \hline
$17$ & $(74/5, 1/3, 1064)$ & $(92/5, 1/3, 549)$  & $(193/10, 1/3, 363)$ \\ \hline
$19$ & $(147/10, 1/3, 1103)$ & $(183/10, 1/3, 570)$ & $(96/5, 1/3, 378)$ \\ \hline
$23$ & $(72/5, 1/3, 1171)$ &   $(181/10, 1/3, 606)$ & $(191/10, 1/3, 404)$ \\ \hline
$29$ & $(71/5, 1/3, 1254)$ &   $(179/10, 1/3, 650)$ & $(19, 1/3,437)$ \\ \hline
$31$ & $(141/10, 1/3, 1279)$ &  $(89/5, 1/3, 662)$ & $(189/10, 1/3, 446)$ \\ \hline
$37$ & $(14, 1/3, 1344)$ &    $(177/10, 1/3, 696)$ & $(94/5, 1/3, 471)$\\ \hline
$41$ & $(139/10, 1/3, 1382)$ & $(88/5, 1/3, 716)$ & $(94/5, 1/3, 486)$\\ \hline
$43$ & $(139/10, 1/3, 1400)$ & $(88/5, 1/3, 726)$ &$(94/5, 1/3, 493)$\\ \hline
$47$ & $(69/5, 1/3, 1433)$ & $(35/2, 1/3, 743)$ & $(187/10, 1/3, 505)$\\ \hline
$53$ & $(137/10, 1/3, 1479)$ & $(35/2, 1/3, 766)$ & $(187/10, 1/3, 523)$\\ \hline
$59$ & $(68/5, 1/3, 1519)$ & $(87/5, 1/3, 787)$ & $(93/5, 1/3, 538)$ \\ \hline
$61$ & $(68/5, 1/3, 1532)$ &  $(87/5, 1/3, 794)$  & $(93/5, 1/3, 543)$\\ \hline
$67$ & $(27/2, 1/3, 1568)$ &  $(173/10, 1/3, 812)$ &$(93/5, 1/3, 557)$\\ \hline
$71$ & $(27/2, 1/3, 1590)$ &  $(173/10, 1/3, 824)$ &$(93/5, 1/3, 565)$\\ \hline
$73$ & $(67/5, 1/3, 1601)$ & $(173/10, 1/3, 829)$ &$(93/5, 1/3, 569)$ \\ \hline
$79$ & $(67/5, 1/3, 1631)$ & $(86/5, 1/3, 845)$ & $(37/2, 1/3,581)$\\ \hline
$83$ & $(133/10, 1/3, 1650)$ & $(86/5, 1/3, 855)$& $(37/2, 1/3, 588)$\\ \hline
$89$ & $(133/10, 1/3, 1677)$ &  $(86/5, 1/3, 868)$   &$(37/2, 1/3, 598)$\\ \hline
$97$ & $(66/5, 1/3, 1710)$ &  $(171/10, 1/3, 886)$ & $(37/2, 1/3, 611)$\\ \hline
\end{tabular}
\caption{Results from Laurent's Theorem~2}\label{table:thm-2-2}
\end{table}

\subsection{Continued fractions}
\label{continued fractions 2}

By increasing the lower bound on $y$ even further, we can reduce the upper bound on $p$ more. An efficient way to do this is using continued fractions. 
Using similar arguments as in Section~\ref{continued fractions 1}, we obtain the following results (variants of Lemmas~\ref{lem:Y0} and \ref{lem:cf-LB}). 

\begin{lem}\label{lem:Y0-2}
If $(x,y)$ is a solution of \eqref{eq:Thue2} with $\gcd(x,y)=1$ and $|x| \geq 2$, then $y/x$ is a convergent in the continued fraction expansion of $\left(\dfrac{k+2}{2}\right)^{1/p}$.
\end{lem}

\begin{lem}
\label{lem:cf-LB-2}
For each prime $3 \leq k \leq 97$ or $k\in \{9,15\}$ and each prime, $p$, satisfying $7 \leq p \leq P_{k,1001}$,
there are no solutions of \eqref{eq:Thue2} with $|x|<10^{100,000}$.
\end{lem}

Applying Laurent's theorem again (using our Maple function \verb+check_1()+ in
\verb+Laurent-Thm-2-estimate.txt+), we obtain the bounds $p \leq P_{k,10^{100,000}}$
in Table~\ref{table:thm-2-2}.

\section{The case of $s = 2k + 4$, with $k$ a prime} \label{calculations for s even}

\subsection{A general setting} \label{subsect:2q+4 q prime}

The equation~\eqref{eq:polygonal} with $s=2k+4$ takes the form 
\begin{equation}
\label{eq:polygonal 2k+4} 
n((k+1)n - k) = t^{p}, \quad n,t \in \bbZ, p>2. 
\end{equation}
If $k$ is an odd prime, then $\gcd(n,(k+1)n-k) = 1$ or $k$.

(i) Assume that $k \nmid n$. In this case there are coprime integers $x$ and $y$ such that 

\begin{equation}\label{eq:polygonal even s Case I} 
(k+1)x^{p} - y^{p} = k.
\end{equation} 

(ii) Assume that $k || n$. In this case there are coprime integers $x$ and $y$ such that 

\begin{equation}\label{eq:polygonal even s Case II} 
(k+1)x^{p} - k^{p-2}y^{p} = 1.
\end{equation} 

Multiplying the equation~\eqref{eq:polygonal even s Case II} by $k^{2}$ we are led to
\begin{equation}\label{eq:polygonal even s Case IIA} 
k^{2}(k+1)X^{p} - Y^{p} = k^{2}. 
\end{equation} 

(iii) Assume that $k^{2} | n$. In this case there are coprime integers $x$ and $y$ such that 

\begin{equation}\label{eq:polygonal even s Case III} 
(k+1)k^{p-2}x^{p} - y^{p} = 1.
\end{equation} 

To prove our results about these Thue equations and those in subsequent sections,
we will use the following result of Mignotte \cite{Mi1}.

\begin{lem}
\label{lem:mignotte-binomial-bnd}
Let $F(x,y) = ax^{n} - by^{n}$ be a binary form of degree $n \geq 3$, with
positive integer coefficients $a$ and $b$, $a \neq b$. Put $A = \max \{a, b, 3 \}$
and
\[
\lambda = \log \left( 1 + \frac{\log A}{|\log (a/b)|} \right).
\]

Suppose that $F(x,y)=c$ for any integer $c$ with $y>|x|>0$. Then
\[
n \leq \max \left\{ 3 \log \left( 1.5|c/b| \right), 7400 \frac{\log A}{\lambda} \right\}.
\]
\end{lem}

\begin{proof}
This is Theorem~1 of \cite{Mi1}.
\end{proof}

\subsection{Presentation of the calculations} \label{subsect:calculations:2q+4 q odd prime}

\begin{lem}
\label{lem: even s Case I}
Let $5 \leq k \leq 97$ be a prime.

\noindent
{\rm (a)}
The equation~\eqref{eq:polygonal even s Case I} has a solution $(x,y) = (1,1)$
for any prime $p$. 

\noindent
{\rm (b)}
Let $3 \leq p \leq 7$. The equation~\eqref{eq:polygonal even s Case I} has other
non-zero solutions only for the following two pairs $(k,p)$: $(5,3)$ (solution
$(x,y) = (-257,-467)$); $(13,3)$ (solutions $(x,y) = (-2,-5), (-1,-3)$).

\noindent
{\rm (c)} If we assume GRH, then the equation~\eqref{eq:polygonal even s Case I}
has no further solutions with $11\le p\le 31$.
\end{lem}

\begin{proof}
(a) is easy, and (b) follows by using the Thue equation solver in $\mathtt{MAGMA}$
\cite{Magma}.

For (c), we used the \verb+thueinit()+ function in $\mathtt{PARI/GP}$. The computations
with $p = 31$ were done on a supercomputer with 220 cores and 300 GB memory for
one day. Since the supercomputer resources were insufficient for the case $p\ge 37$
(and all primes $5\le k \le 97$), we stopped at $p = 31$.
\end{proof}

\begin{rem-nonum}
(i) We will use this idea in the proofs of Lemmas~\ref{lem: odd s Case Ia}, \ref{lem: odd s Case Ib},
Section~\ref{subsect:calculations:s=13} (i), Section~\ref{subsect:calculations:s=16} (i),(ii),(iii),(vii),
and Section~\ref{subsect:calculations:s=19} (i),(vii).

(ii) Let $5 \leq k \leq 97$, and $11\le p\le 31$ be primes. For the Thue equations~\eqref{eq:polygonal even s Case I},
the class number $h$ is always one (under GRH) except for the following cases:
$(p,k,h) \in \{(11,37,2), (11,59,11), (19,73,19)\}$. 
\end{rem-nonum}

\begin{lem}
\label{lem: even s Case II} 
Let $5 \leq k \leq 97$ be a prime.

\noindent
{\rm (a)}
The equation~\eqref{eq:polygonal even s Case II} has a unique solution $(x,y) = (1,1)$ for $p = 3$. 

\noindent
{\rm (b)}
The equation~\eqref{eq:polygonal even s Case II} has no solution for any prime $p \geq 5$.
\end{lem} 

\begin{proof}
We proceed in four steps. For the first three steps, we work instead with equation~\eqref{eq:polygonal even s Case IIA}.

(i) For each $k$, we start by using Lemma~\ref{lem:mignotte-binomial-bnd}
to obtain an upper bound on $p$ such that \eqref{eq:polygonal even s Case IIA}
may have a non-trivial integer solution.

In the notation of Lemma~\ref{lem:mignotte-binomial-bnd}, we have $a=k^{2}(k+1)$,
$b=1$ and $c=k^{2}$. So $A=k^{2}(k+1)=a/b$. Hence $\lambda=2$
and $p<10,700 \log \left( k^{2}(k+1) \right)$.

\vspace*{1.0mm}

(ii) For each prime $k$ with $5 \leq k \leq 97$ and all primes $p < 10,700 \log \left( k^{2}(k+1) \right)$,
we applied a modular arithmetic argument (a Sophie Germain type argument) like
that in the proof of Theorem~1 in \cite{GP} (see the opening paragraph on page~492 there)
to the equation~\eqref{eq:polygonal even s Case IIA}, except that we do not do
the modulo $n^{2}$ check there (modulo $p^{2}$ in our notation here).

This entailed checking if \eqref{eq:polygonal even s Case IIA} is solvable modulo $q$
for all primes $q=2pr+1$, with $1 \leq r \leq 150$. Rather than checking if this
equation is solvable for all $q^{2}$ pairs $(x,y)$ modulo $q$, we note that $u \neq 0$
is a $p$-th power modulo $q$ if and only if $u^{(q-1)/p} \equiv 1 \pmod{q}$
(see, for instance Proposition~4.2.1 of \cite{IR}).
There are at most $2r$ such values of $u$, so together with $0$, we only need to
check at most $(2r+1)^{2}$ possible values rather than $q^{2}$ possible values
for each congruence.

This eliminated all but $47$ of the pairs $(k,p)$. With the exceptions of $(k,p)=(29,19)$,
$(k,p)=(67,11)$ and $(k,p)=(67,19)$, we had $p \leq 7$ for these remaining pairs.
$\mathtt{PARI/GP}$  code took 36 seconds on PV's laptop.

\vspace*{1.0mm}

(iii) We use the $\mathtt{MAGMA}$  \cite{Magma} function \verb!HasPointsEverywhereLocally()!
to eliminate some of the remaining pairs $(k,p)$ after step~(ii).

The check in step~(ii) is done first, as step~(iii) can be expensive time-wise,
especially as $p$ grows, since we have to check local solvability of the equations
for all primes $q$ less than $O \left( p^{4} \right)$ -- see, for instance,
Theorem~6.4.2 in \cite{Cohen}.

This left $45$ pairs $(k,p)$. Most importantly, the pair $(k,p)=(29,19)$ is
eliminated (it has no solutions in $\bbQ_{19}$).

Note that the function, \verb!HasPointsEverywhereLocally()!, checks for
solvability of the homogenisation of our Thue equations, rather than the Thue
equations themselves. For some pairs, $(k,p)$, this can result in local solvability
of the homogenised equation, while the Thue equations themselves do not have
local solutions. E.g., $(k,p)=(5,17)$, where the homogenised equation has the
solution $(x,y,z)=(1,57,0)$ modulo $103$.

\vspace*{1.0mm}

(iv) We solved the equations of the form \eqref{eq:polygonal even s Case II} 
for the remaining pairs $(k,p)$ using the Thue equation solver in $\mathtt{PARI/GP}$. The
benefit of using equation~\eqref{eq:polygonal even s Case II} is that the
right-hand side of the Thue equation is $1$, so we get unconditional results
from the \verb!thueinit()! function in $\mathtt{PARI/GP}$ without setting the flag argument
to that function to be non-zero. This makes the calculation much faster, and
hence feasible. See \url{https://pari.math.u-bordeaux.fr/dochtml/ref-stable/Polynomials_and_power_series.html#thueinit}
for details.

With the exception of $(k,p)=(67,19)$,
we had $3 \leq p \leq 7$ for the remaining $45$ pairs after step~(iii).

This calculation (excluding $(k,p)=(67,19)$) took 60 seconds on PV's laptop, while the case $(67,19)$ was solved on a supercomputer.
\end{proof}

\begin{lem} \label{lem: even s Case III}
Let $5 \leq k \leq 97$ be a prime.

\noindent
{\rm (a)} The equation~\eqref{eq:polygonal even s Case III} has a solution $(x,y) = (0,-1)$ for any prime $p$. 

\noindent
{\rm (b)} Let $3 \leq p \leq 23$. The equation~\eqref{eq:polygonal even s Case III} has a non-zero solution only 
for $(k,p) = (13,3)$, namely $(x,y) = (3,17)$. 

\noindent
{\rm (c)} The equation~\eqref{eq:polygonal even s Case III} has no non-zero solution 
for any $p\geq 29$, with possible exceptions $(k,p) = (k,k)$ (with $29 \leq p \leq 97$ and $p\neq 31$)
or $(k,p) \in \{(59,29), (83,41) \}$.

\noindent
{\rm (d)} The weak effective $abc$ conjecture ($abc(r)$) implies that the equation~\eqref{eq:polygonal even s Case III} has no solutions for $r= 1.63$ when $(k,p) = (k,k)$ (with $29 \leq p \leq 97$ and $p\neq 31$)
or $(k,p) \in \{(59,29), (83,41) \}$.
\end{lem}
    
\begin{proof}
Here we first apply Theorem~1.3 of Bartolom\'{e} and Mih\u{a}ilescu \cite{BM}, which reduces the problem to 55 pairs.
Next, we use the Thue equation solver in $\mathtt{MAGMA}$  for the case $p\le 7$.
Finally, for the remaining 30 pairs, we employ $\mathtt{PARI/GP}$.
However, $\mathtt{PARI/GP}$ (including supercomputer) succeeds in solving only 12
pairs using \verb+thueinit()+, namely $(k,p)=\{(11,11),(13,13),(17,17),(19,19),(23,11), (23,23),\\ (31,31), (47,23), (53,13), (67,11),(79,13),(89,11)\}$. We expect that there are no non-zero solutions for the remaining cases in part~(c),
but we were unable to confirm this numerically. The case $(k,p)=(47,23)$ was resolved on the supercomputer in 2 hours, 56 minutes, and 54 seconds.
However, for the cases $(k,p)\in\{(59,29),(83,41)\}$ and $(k,p)=(k,k)$ with $29 \leq p \leq 97$ and $p \neq 31$, our supercomputer resources were insufficient, even when using 128-220 cores and 350 GB of memory for 315 CPU days.

For the proof of part~(d), we apply a similar argument as in the proof of Lemma~\ref{lem:effective abc}.
Here $a=(k+1)k^{p-2}x^{p}$, $b=-1$ and $c=y^{p}$. So $L(a,b,c) = \log(c)/ \log \rad(abc)
=p\log(y)/\log \rad \left( k(k+1)xy \right)>p\log(y)/\log \left( k(k+1)xy \right)$. 
\end{proof}

\subsection{Generalisations}
\label{subsect:calculations:s=12}

The general setting presented in Subsection~\ref{subsect:2q+4 q prime} can be
easily generalised to the case $s = 2q^{r} + 4$. Let us consider in some detail
the special case $s = 12$. In this case, the equation~\eqref{eq:polygonal} takes
the form 
\begin{equation}
\label{eq:polygonal12} 
n(5n-4) = t^{p}, \quad n,t \in \bbZ, p>2. 
\end{equation}

(i) Assume that $n$ is odd. In this case there exist coprime integers
$x$ and $y$ such that $5x^p-y^p=4$. For $3\le p \le 13$ we use the Thue
equation solver in $\mathtt{MAGMA}$ \cite{Magma} to check that $(x, y) = (1, 1)$ is the only
solution to this equation. But for the case $17\le p \le 53$, we use
$\mathtt{PARI/GP}$ (thueinit, flag=0) (including supercomputer), to check that
$(x, y) = (1, 1)$ is the only solution to this equation (these computations are
under GRH). The case $p\ge 59$ is discussed in Section~\ref{sect:p-power-even}.

(ii) Assume that $2|| n$. In this case, there exist coprime integers $x$ and $y$ with $2 \nmid x$ such that 
$5x^{p} - 2^{p-2} y^{p} = 2$, which is impossible for $p > 2$ and odd $x$. 

(iii) Assume that $2^{2} || n$. In this case, there exist coprime integers $x$ and $y$ with $2 \nmid x$ such that 
$5x^{p} - 2^{p-4} y^{p} = 1$. 
Fortunately, this equation is a particular case of a Diophantine equation which has been already solved using 
the multi-Frey approach by Bugeaud, Mignotte and Siksek (see \cite[Theorem~1.1]{BMS}). As a consequence, we find that the only solution 
to the above equation is $(x,y,p) = (1,2,3)$, which corresponds to the solution $(n,t,p)=(4,4,3)$ of equation~\eqref{eq:polygonal12}.

(iv) Assume that $2^{3} | n$. In this case, there exist coprime integers $x$ and $y$ such that 
$5\cdot 2^{p-4}x^{p} - y^{p} = 1$. 
Fortunately, this equation is a particular case of a Diophantine equation which has been already solved
by Bennett, Gy\H{o}ry, Mignotte and Pint\'{e}r (see \cite[Theorem~1.1]{BGMP}). We obtain that the only solution to the above equation is 
$(x,y,p) = (14,19,3)$, which corresponds to the solution $(n,t,p)=(5488,532,3)$ of equation~\eqref{eq:polygonal12}.

\section{The case of $s = k + 4$, with $k$ an odd prime}  \label{calculations for s odd} 

\subsection{A general setting} \label{subsect:q+4 q odd prime}

The equation~\eqref{eq:polygonal} with $s=k+4$ takes the form 
\begin{equation}
\label{eq:polygonal k+4} 
n((k+2)n - k) = 2t^{p}, \quad n,t \in \bbZ, p>2. 
\end{equation}
If $k$ is an odd prime, then $\gcd(n,(k+2)n-k) = 1$ or $k$. 

(i) Assume that $k \nmid n$. In this case there are coprime integers $x$ and $y$ such that 

\begin{equation}\label{eq:polygonal odd s Case Ia} 
(k+2)x^{p} - 2y^{p} = k
\end{equation} 
or 
\begin{equation}\label{eq:polygonal odd s Case Ib} 
2(k+2)x^{p} - y^{p} = k.
\end{equation} 

(ii) Assume that $k || n$. In this case there are coprime integers $x$ and $y$ such that 

\begin{equation}\label{eq:polygonal odd s Case IIa} 
(k+2)x^{p} - 2\cdot k^{p-2}y^{p} = 1
\end{equation} 
or 
\begin{equation}\label{eq:polygonal odd s Case IIb} 
2(k+2)x^{p} - k^{p-2}y^{p} = 1.
\end{equation}

Multiplying the equation~\eqref{eq:polygonal odd s Case IIa} by $k^{2}$ we are led to 
\begin{equation}\label{eq:polygonal odd s Case IIA} 
k^{2}(k+2)X^{p} - 2Y^{p} = k^{2}.
\end{equation}

Multiplying the equation~\eqref{eq:polygonal odd s Case IIb} by $k^{2}$ we are led to 
\begin{equation}\label{eq:polygonal odd s Case IIB} 
2k^{2}(k+2)X^{p} - Y^{p} = k^{2}.
\end{equation}

(iii) Assume that $k^{2} | n$. In this case there are coprime integers $x$ and $y$ such that 
\begin{equation}\label{eq:polygonal odd s Case IIIa} 
(k+2)\cdot k^{p-2}x^{p} - 2y^{p} = 1
\end{equation}
or 
\begin{equation}\label{eq:polygonal odd s Case IIIb} 
2(k+2)\cdot k^{p-2}x^{p} - y^{p} = 1.
\end{equation}

Multiplying the equation~\eqref{eq:polygonal odd s Case IIIa} by $k^{2}$ we are led to 
\begin{equation}\label{eq:polygonal odd s Case IIIA} 
(k+2)X^{p} - 2k^{2}Y^{p} = k^{2}.
\end{equation}

\subsection{Presentation of the calculations}
\label{subsect:calculations:q+4 q odd prime}

\begin{lem}
\label{lem: odd s Case Ia}
Let $3 \leq k \leq 97$ be a prime.

\noindent
{\rm (a)} The equation~\eqref{eq:polygonal odd s Case Ia} has a solution $(x,y) = (1,1)$ for any prime $p$.

\noindent
{\rm (b)} Let $3 \leq p \leq 7$. The equation~\eqref{eq:polygonal odd s Case Ia} has other non-zero solutions only
for the following two pairs $(k,p)$:
$(7,3)$ (solutions $(x,y) = (-3,-5), (-1,-2)$); $(31,5)$ (solution $(x,y) = (-1,-2)$).

\noindent
{\rm (c)} If we assume GRH, then the equation~\eqref{eq:polygonal odd s Case Ia} has no further solutions with $11\le p\le 29$.
\end{lem}

\begin{proof}
The proof follows the steps of the proof of Lemma~\ref{lem: even s Case I} using a supercomputer.
\end{proof}

\begin{lem} \label{lem: odd s Case Ib} 
Let $3 \leq k \leq 97$ be a prime.
 
\noindent
{\rm (a)} If $(k,p)\not\in\{(5,13),(23,13),(29,11),(79,19), (83,19),(89,19)\}$, then the equation~\eqref{eq:polygonal odd s Case Ib} has non-zero solutions only for $(k,p)\in\{(23,3),(41,3)\}$, namely $(x,y)\in\{(1,3),(-2,-9)\}$.

\noindent
{\rm (b)} If we assume GRH, then equation~\eqref{eq:polygonal odd s Case Ib} has no non-zero solutions with $(k,p)\in\{(5,13),(23,13),(29,11),(79,19), (83,19),(89,19) \}$.
\end{lem} 
\begin{proof}
Following the proof of Lemma~\ref{lem: even s Case II}, we complete the proof. 
\end{proof}
\begin{lem} \label{lem: odd s Case IIa} 
Let $3 \leq k \leq 97$ be a prime.

\noindent
{\rm (a)} If $p=3$, then the equation~\eqref{eq:polygonal odd s Case IIa} has a solution only for 
$k=3$ or $k=83$, namely $(x,y) = (-1,-1)$ or $(x,y) = (5,4)$, respectively.

\noindent
{\rm (b)} The equation~\eqref{eq:polygonal odd s Case IIa} has no solution for $p \geq 5$.
\end{lem} 
\begin{proof}
Using the proof steps of Lemma~\ref{lem: even s Case I}, we resolve (a) and (b).
Here, the supercomputer runs for $(k,p)\in \{ (71,13),(73,17) \}$ took 1h35m31s and 1h06m08s, respectively.
For $(k,p)=(59,19)$ we used $\mathtt{PARI/GP}$ version 2.16, git branch \textit{bill-parbnf-gokhan}: \url{http://pari.math.u-bordeaux.fr/git.html}.
It took 26.5 hours (95408424 ms) with 224 cores and 300 GB of memory.
\end{proof}
\begin{lem} \label{lem: odd s Case IIb} 
Let $3 \leq k \leq 97$ be a prime. The equation~\eqref{eq:polygonal odd s Case IIb} has a non-zero solution only for $(k,p)=(3,3)$, namely $(x,y)=(-2,-3)$.
\end{lem} 

\begin{proof}
The proof follows the steps in Lemma~\ref{lem: even s Case II}. The case $(k,p)=(97,19)$
was computed on a supercomputer taking 6 days, 11 hours, 36 minutes, and 23 seconds,
while $(k,p)=(71,19)$ was handled using $\mathtt{PARI/GP}$ version 2.16, git branch
\textit{bill-parbnf-gokhan}: \url{http://pari.math.u-bordeaux.fr/git.html} on the
supercomputer. It took 48.9 hours (176056466 ms) with 224 cores and 300 GB memory.
\end{proof}

\begin{lem} \label{lem: odd s Case IIIa} 
Let $3 \leq k \leq 97$ be a prime.

\noindent
{\rm (a)} If $p=3$, then the equation~\eqref{eq:polygonal odd s Case IIIa} has a solution only
for $k = 3$ $($one solution $(x,y) = (-1,-2))$, for $k = 23$ $($one solution $(x,y) = (5,33))$, for $k = 31$ $($one solution $(x,y) = (-1,-8))$. 

\noindent
{\rm (b)} The equation~\eqref{eq:polygonal odd s Case IIIa} has no solution for $p\geq 5$.
\end{lem}

\begin{proof}
For (a), we use the Thue solver function in $\mathtt{PARI/GP}$. Next, using the
Sophie Germain type argument with $\mathtt{MAGMA}$ we continue (b). This left 35
pairs. Finally, treating them with Thue solver in $\mathtt{PARI/GP}$, we complete
the proof.
\end{proof}

\begin{lem} \label{lem: odd s Case IIIb} 
Let $3 \leq k \leq 97$ be a prime.

\noindent
{\rm (a)} The equation~\eqref{eq:polygonal odd s Case IIIb} has a solution $(x,y) = (0,-1)$ for any prime $p$.

\noindent
{\rm (b)} The equation~\eqref{eq:polygonal odd s Case IIIb} has no non-zero solutions for all odd primes $p$, 
with possible exceptions $(k,p) = (k,k)$ (with $29 \leq p \leq 97$) or $(k,p) \in
\{(59,29), (83,41) \}$.

\noindent
{\rm (c)} The weak effective $abc$ conjecture ($abc(r)$) implies that the
equation~\eqref{eq:polygonal odd s Case IIIb} has no solutions for $r= 1.63$
when $(k,p) = (k,k)$ (with $29 \leq p \leq 97$) or
$(k,p) \in \{(59,29), (83,41) \}$.
\end{lem} 

\begin{proof}
The proof proceeds as in Lemma~\ref{lem: even s Case III}. Cases $(k,p)\in\{(19,19),\\(23,23),(47,23)\}$
were computed on a supercomputer with 220 cores and 300 of GB memory for one day
using $\mathtt{PARI/GP}$. However, for $(k,p)\in\{(59,29),(83,41)\}$, and for
$p\ge 29$, the available resources (up to 128 cores, 1 TB of memory, and 15 days)
were insufficient.

For the proof of part~(c), we apply a similar argument as in the proof of Lemma~\ref{lem:effective abc}.
Here $a=2(k+2)k^{p-2}x^{p}$, $b=-1$ and $c=y^{p}$. So $L(a,b,c) = \log(c)/ \log \rad(abc)
=p\log(y)/\log \rad \left( 2(k+2)xy \right)>p\log(y)/\log \left( 2(k+2)xy \right)$.
\end{proof} 

\subsection{Generalisations}
\label{subsect:calculations:s=13}

The general setting presented in Subsection~\ref{subsect:q+4 q odd prime} can be
easily generalised to the case $s = q^{r} + 4$, with $q$ an odd prime. Let us
consider in some detail a special case $s = 13$. In this case, the equation~\eqref{eq:polygonal}
takes the form 
\begin{equation}
\label{eq:polygonal7} 
n(11n-9) = 2 t^{p}, \quad n,t \in \bbZ, p>2. 
\end{equation}

Using the Sophie Germain type argument, applying \verb!HasPointsEverywhereLocally()!
(if necessary) and the Thue equation solver in $\mathtt{MAGMA}$ \cite{Magma}, we
completely solve the Diophantine equations from (iii) below, the second equation
from item~(i) below and the first equation from item~(iv) below. For the second
Diophantine equation in (iv) below, we apply the Bartolom\'{e}-Mih\u{a}ilescu theorem
and the Thue equation solver in $\mathtt{MAGMA}$.

(i) Assume that $3 \nmid n$. In this case there are coprime integers $x$ and $y$
such that $11x^{p} - 2y^{p} = 9$ or $22x^{p} - y^{p} = 9$. For $3 \leq p < 11$,
$\mathtt{MAGMA}$ confirms that $(x,y)=(1,1)$ is the only solution. For $11 \leq p \leq 37$,
computations are carried out with $\mathtt{PARI/GP}$ on a laptop, while the cases
$p\in\{41,43\}$ require a supercomputer (these computations are under GRH). For
$p>43$, the available memory is insufficient, and the computations were terminated.

For a prime $p > 43$, the weak effective $abc$ conjecture ($abc(r)$) with $r=1.63$ implies
that $(x, y) = (1, 1)$ is the only solution of the equations $11x^p - 2y^p = 9$
(use Lemma~\ref{lem:effective abc 2}) and $22x^p - y^p = 9$ (use a similar argument).

(ii) Assume that $3 || n$. In this case there are coprime integers $x$ and $y$ such that 
$11x^{p} - 2\cdot 3^{p-2}y^{p} = 3$ or $22x^{p} - 3^{p-2}y^{p} = 3$.
It is plain to see that both equations have no non-trivial solutions for $p\geq 3$. 

(iii) Assume that $3^{2} || n$. In this case there are coprime integers $x$ and $y$ such that 
$11x^{p} - 2\cdot 3^{p-4}y^{p} = 1$ or $22x^{p} - 3^{p-4}y^{p} = 1$. These equations have no solutions.

(iv) Assume that $3^{3} | n$. In this case there are coprime integers $x$ and $y$ such that
$11\cdot 3^{p-4}x^{p} - 2y^{p} = 1$ or $22\cdot 3^{p-4}x^{p} - y^{p} = 1$. The first equation has no solutions, 
and the second one has the only solution $(x,y) = (0,-1)$.

\section{The cases $s = 16, 19$}
\label{sect:cases s = 16, 19} 

Neither of the values $s = 16, 19$ is of the form $2q^{r}+4$ or $q^{r}+4$ considered above.
Below we consider the equation~\eqref{eq:polygonal} for these two values of $s$ in some
detail, give a summary of the results for the equation~\eqref{eq:polygonal} for all
$5 \leq s \leq 20$, and propose some questions.

\subsection{The case $s = 16$} \label{subsect:calculations:s=16}

In this case, the equation~\eqref{eq:polygonal} takes the form 
\begin{equation} \label{eq:polygonal16} 
n(7n-6) = t^{p}, \quad n,t \in \bbZ, p>2.
\end{equation}

Using the methods of Subsection~\ref{subsect:calculations:s=13} (Sophie Germain type argument
and the Thue solvers in $\mathtt{MAGMA}$ and $\mathtt{PARI/GP}$ with \verb+thueinit+, flag=0)
for (ii)--(viii), and Theorem~1.3 of Bartolom\'{e} and Mih\u{a}ilescu for (ix), we completely solve the corresponding Diophantine equations.

(i) Assume that $6 \nmid n$. In this case, there exist coprime integers $x$ and $y$ such that $7x^{p} - y^{p} = 6$. 
For $3 \leq p \leq 11$, $\mathtt{MAGMA}$ verifies that $(x,y)=(1,1)$ is the unique solution to this equation, while for $13 \leq p \leq 41$ we employ $\mathtt{PARI/GP}$ ($\mathtt{PARI/GP}$ computations are
under GRH). The case $p\geq 43$ is discussed in Section~\ref{sect:p-power-even}.

(ii) Assume that $2 \nmid n$ and $3||n$. In this case, there exist coprime integers $x$ and $y$ such that 
$7x^{p} - 3^{p-2}y^{p} = 2$. This equation has no solution. The only case  $p=11$ is resolved by $\mathtt{PARI/GP}$ (this computation is under GRH).

(iii) Assume that $2 \nmid n$ and $3^{2}|n$. In this case, there exist coprime integers $x$ and $y$ such that 
$7 \cdot 3^{p-2}x^{p} - y^{p} = 2$. This equation has no solution. $\mathtt{PARI/GP}$ resolves the case $p=19$ (this computation is under GRH).

(iv) Assume that $2 || n$ and $3 \nmid n$. In this case, there exist coprime integers $x$ and $y$ such that 
$7x^{p} - 2^{p-2}y^{p} = 3$. This equation has no solution.

(v) Assume that $2|| n$ and $3||n$. In this case, there exist coprime integers $x$ and $y$ such that 
$7x^{p} - 6^{p-2}y^{p} = 1$. This equation has the only solution $(x,y,p) = (1,1,3)$, which corresponds to the solution 
$(n,t,p)=(6,6,3)$ of equation~\eqref{eq:polygonal16}. Note that the above equation is a particular case of 
a Diophantine equation which has been already solved using the multi-Frey approach by Bugeaud, Mignotte and Siksek 
(see \cite[Theorem~1.3]{BMS}).

(vi) Assume that $2|| n$ and $3^{2}|n$. In this case, there exist coprime integers $x$ and $y$ such that 
$7 \cdot 3^{p-2}x^{p} - 2^{p-2}y^{p} = 1$. This equation has no solution.

(vii) Assume that $2^{2}| n$ and $3\nmid n$. In this case, there exist coprime integers $x$ and $y$ such that 
$7 \cdot 2^{p-2}x^{p} - y^{p} = 3$. This equation has no solution. We only use $\mathtt{PARI/GP}$ for the case $p=11$ (this computation is under GRH).

(viii) Assume that $2^{2}| n$ and $3||n$. In this case, there exist coprime integers $x$ and $y$ such that 
$7 \cdot 2^{p-2}x^{p} - 3^{p-2}y^{p} = 1$. This equation has no solution. 

(ix) Assume that $2^{2}| n$ and $3^{2}|n$. In this case, there exist coprime integers $x$ and $y$ such that 
$7 \cdot 6^{p-2}x^{p} - y^{p} = 1$. This equation has the only solution $(x,y) = (0,-1)$.

\subsection{The case $s = 19$} \label{subsect:calculations:s=19}

In this case, the equation~\eqref{eq:polygonal} takes the form 
\begin{equation}
\label{eq:polygonal19} 
n(17n-15) = 2 t^{p}, \quad n,t \in \bbZ, p>2. 
\end{equation}

Applying the methods of Subsection~\ref{subsect:calculations:s=13} for (ii)-(viii)
and the first equation of (ix), and Theorem~1.3 of Bartolom\'{e} and Mih\u{a}ilescu for the second equation of (ix), we completely solve the following Diophantine equations.

(i) Assume that $15 \nmid n$. In this case there are coprime integers $x$ and $y$ such that 
$17x^{p} - 2y^{p} = 15$ or $34x^{p} - y^{p} = 15$.
For $3 \leq p < 11$ we use $\mathtt{MAGMA}$, and for $11 \leq p \leq 31$ we use $\mathtt{PARI/GP}$. The case $p=37$ was solved on a supercomputer, whereas for $p \geq 41$ the available memory proved insufficient ($\mathtt{PARI/GP}$ computations are under GRH). The case $p \geq 41$ is discussed in Section~\ref{sect:p-power-odd}. The second equation has no solutions. 

For a prime $p > 37$, the weak effective $abc$ conjecture ($abc(r)$) with $r=1.63$
implies that $(x, y) = (1, 1)$ is the only solution of the equations $17x^p - 2y^p = 15$
(use Lemma~\ref{lem:effective abc 2}) and $34x^p - y^p = 15$ (use a similar argument).

(ii) Assume that $3 \nmid n$ and $5 || n$. In this case there are coprime integers
$x$ and $y$ such that $17x^{p} - 2\cdot 5^{p-2}y^{p} = 3$ $(p\neq 37,41)$ or
$34x^{p} - 5^{p-2}y^{p} = 3$. These equations have no solutions. Here, the cases
$p=37,41$ were considered with 220 cores, 300 GB of memory for 3 days. But these cases
were unsuccessful due to insufficient memory.

The weak effective $abc$ conjecture ($abc(r)$) with $r=1.63$ implies that the first
equation has no solution for $p = 37$ or $41$.

(iii) Assume that $3 \nmid n$ and $5^{2} | n$. In this case there are coprime
integers $x$ and $y$ such that $17 \cdot 5^{p-2}x^{p} - 2y^{p} = 3$ or
$34 \cdot 5^{p-2}x^{p} - y^{p} = 3$. These equations have no solutions. 

(iv) Assume that $3 || n$ and $5 \nmid n$. In this case there are coprime integers
$x$ and $y$ such that $17x^{p} - 2\cdot 3^{p-2}y^{p} = 5$ or $34x^{p} - 3^{p-3}y^{p} = 5$.
These equations have no solutions. 

(v) Assume that $3 || n$ and $5 || n$. In this case there are coprime integers $x$ and $y$ such that
$17x^{p} - 2\cdot 15^{p-2}y^{p} = 1$ or $34x^{p} - 15^{p-2}y^{p} = 1$. These equations have no solutions. 

(vi) Assume that $3 || n$ and $5^{2} | n$. In this case there are coprime integers $x$ and $y$ such that
$17\cdot 5^{p-2}x^{p} - 2\cdot 3^{p-2}y^{p} = 1$ or $34\cdot 5^{p-2}x^{p} - 3^{p-2}y^{p} = 1$. 
These equations have no solutions.

(vii) Assume that $3^{2} | n$ and $5 \nmid n$. In this case there are coprime integers $x$ and $y$ such that
$17\cdot 3^{p-2}x^{p} - 2y^{p} = 5$ or $34\cdot 3^{p-2}x^{p} - y^{p} = 5$. These equations have no solutions. We only use $\mathtt{PARI/GP}$ for the case $p\in\{11,19\}$ of the first equation (these computations are under GRH).

(viii) Assume that $3^{2} | n$ and $5 || n$. In this case there are coprime integers $x$ and $y$ such that 
$17\cdot 3^{p-2}x^{p} - 2\cdot 5^{p-2}y^{p} = 1$ or $34\cdot 3^{p-2}x^{p} - 5^{p-2}y^{p} = 1$. 
These equations have no solutions.

(ix) Assume that $3^{2} | n$ and $5^{2} | n$. In this case there are coprime integers $x$ and $y$ such that 
$17\cdot 15^{p-2}x^{p} - 2y^{p} = 1$ or $34\cdot 15^{p-2}x^{p} - y^{p} = 1$. The first equation has no solutions, 
and the second one has the only solution $(x,y) = (0,-1)$.

For $\mathtt{MAGMA}$ computations, we used a MacBook Pro computer with the
following characteristics: Processor M2 Pro, i12, 3.2 GHz, 16~GB RAM, 1 TB SSD.
All these computations were done without using GRH with MAGMA V2.28-14.

\appendix 

\section{Linear forms in logs revisited} 
\label{sect:A-1} 

We can improve the upper bounds on $p$ obtained in Sections~\ref{sect:p-power-even} and \ref{sect:p-power-odd}
further by using Laurent's Theorem~1 instead of his Theorem~2. 


\begin{lem}[Laurent]
\label{lem:Laurent-Thm1}
Let $K$, $L$, $R_{1}$, $R_{2}$, $S_{1}$ and $S_{2}$ be positive integers with
$K \geq 2$. Let $\varrho$ and $\mu$ be real numbers with $\varrho>1$ and
$1/3 \leq \mu \leq 1$. Put
\begin{align*}
R=R_{1}+R_{2}-1, \hspace*{1.0mm} S=S_{1}+S_{2}-1, \hspace*{1.0mm} N=KL, \hspace*{1.0mm}
g = \frac{1}{4}-\frac{N}{12RS}, \\
\sigma = \frac{1+2\mu-\mu^{2}}{2}, \hspace*{1.0mm}
b= \frac{(R-1)b_{2}+(S-1)b_{1}}{2} \left( \prod_{k=1}^{K-1} k! \right)^{-2/(K^{2}-K)}.
\end{align*}

Let $a_{1}$ and $a_{2}$ be positive real numbers such that
\[
a_{i} \geq \varrho \left| \log \alpha_{i} \right| - \log \left| \alpha_{i} \right| + 2D h \left( \alpha_{i} \right)
\]
for $i=1,2$. Suppose that
\begin{align}
\label{eq:L-1}
\card \left\{ \alpha_{1}^{r}\alpha_{2}^{s}: 0 \leq r< R_{1}, 0 \leq s<S_{1} \right\} & \geq L, \\
\card \left\{ rb_{2}+sb_{1}: 0 \leq r< R_{2}, 0 \leq s<S_{2} \right\} & > (K-1)L \nonumber
\end{align}
and
\begin{align}
\label{eq:L-2}
& K(\sigma L - 1) \log \varrho - (D+1)\log N - D(K-1)\log b \\
& -gL \left( Ra_{1}+Sa_{2} \right) > \varepsilon(N), \nonumber
\end{align}
where
\[
\varepsilon(N) = 2 \log \left( N! N^{-N+1} \left( e^{N}+(e-1)^{N} \right) \right) / N.
\]

Then
\[
\left| \Lambda' \right| > \varrho^{-\mu KL}
\hspace*{1.0mm} \text{ with } \hspace*{1.0mm} \Lambda' = \Lambda \max \left\{ \frac{LSe^{LS|\Lambda|/(2b_{2})}}{2b_{2}}, \frac{LRe^{LR|\Lambda|/(2b_{1})}}{2b_{1}} \right\}.
\]
\end{lem}

So here with our linear form in \eqref{eq:lfl}, we will apply this result with
$b_{2}=p$, $\alpha_{1}=k+1$, $b_{1}=1$ and $\alpha_{2}=y/x$.

\subsubsection{Choice of parameters}

Now we need to choose the parameters to use. Instead of following Laurent's choice
in the proof of his Theorem~2 in \cite{Laurent}, we proceed as in the kit of Mignotte
and Voutier \cite{MV}.

We let $L$ be an integer and put
\begin{equation}
\label{eq:k-value}
K = \lfloor m L a_{1} a_{2} \rfloor,
\end{equation}
for a positive real number $m$. This is like the expression for $K$ in Section~3.1
of \cite{Laurent}, except we have kept the notation of \cite{MV}.

We define
\begin{align}
\label{eq:RST-values}
R_{1}    &= 1+ \lfloor c_{1} a_{2} \rfloor,
& S_{1}  &= 1+ \lfloor c_{1} a_{1} \rfloor, \\
R_{2}    &= 1+ \lfloor c_{2} a_{2} \rfloor,
& S_{2}  &= 1+ \lfloor c_{2} a_{1} \rfloor, \nonumber
\end{align}
where the parameters $c_{1}$ and $c_{2}$ will be chosen so that the conditions in
\eqref{eq:L-1} are satisfied.
The motivation for this choice of these quantities is so that both terms in
$a_{1}R+a_{2}S$ on the left-hand side of equation~\eqref{eq:L-2} are roughly
the same size, $O \left( a_{1}a_{2} \right)$, and so that the $gL \left( a_{1}R+a_{2}S \right)$
term on the left-hand side of \eqref{eq:L-2} is roughly the same size as the other
main terms on the right-hand side of \eqref{eq:L-2}, $D(K-1)\log b$.

\vspace*{3.0mm}

\noindent
$\bullet$ $c_{1}$.
We consider the first condition in \eqref{eq:L-1}. If $\alpha_{1}$ and $\alpha_{2}$
are multiplicatively independent, then the left-hand side is
$R_{1}S_{1}>c_{1}^{2}a_{1}a_{2}$. So the first condition in \eqref{eq:L-1} holds
if $c_{1} \geq c_{1}'=\sqrt{L/ \left( a_{1}a_{2} \right)}$. Put our initial values
of $R_{1}$ and $S_{1}$ to be those obtained from \eqref{eq:RST-values} with $c_{1}$
replaced by $c_{1}'$.

It will turn out that $S_{1}$ is very small in our work. In fact, we will typically
have $S_{1}=1$, using $c_{1}'$. So we can take $c_{1}$ to be smaller than $c_{1}'$.
From the first condition in \eqref{eq:L-1},
put $c_{1}=\max \left\{ \left( S_{1}-1 \right)/a_{1}, \left( \lceil L/S_{1} \rceil -1 \right)/ a_{2} \right\}$.
The first term in the max arises because we do not want to change the value of $S_{1}$.

If $\alpha_{1}$ and $\alpha_{2}$ are not multiplicatively independent, then we
use Theorem~1.3 in \cite{Ya} instead. The lower bounds obtained from this result
are better than Lemma~\ref{lem:Laurent-Thm1}. So we can assume multiplicative
independence throughout.

\vspace*{3.0mm}

\noindent
$\bullet$ $c_{2}$.
Similarly, for the second condition in \eqref{eq:L-1}, if there are no linear
relations (we will consider the possibility of linear relations below), then the
left-hand side is $R_{2}S_{2}>c_{2}^{2}a_{1}a_{2}$. So the second condition in
\eqref{eq:L-1} holds if $c_{2}>c_{2}'\sqrt{mL^{2}}$.

For our work, the $S_{i}$'s will be constants for a fixed value of $k$ (i.e.,
not dependent on $y/x$). So we can sometimes take $c_{2}$ to be smaller. This
is useful to us as $c_{2}$ (and the second condition in \eqref{eq:L-1}) is the
more important of the $c_{i}$'s for obtaining good bounds.

First, let $c_{2}''$ be the smallest positive real number such that
$\lfloor c_{2}''a_{1} \rfloor=\lfloor c_{2}'a_{1} \rfloor$.
If the fractional part of $c_{2}'a_{1}$ is near $1$, while its integer part is
small (which can occur for the later iterations), then this can make a
reasonable improvement in the final result obtained.

However, setting $R_{2}=1+\lfloor c_{2}''a_{2} \rfloor$ may violate
the second condition in \eqref{eq:L-1}, so we set
$c_{2}'''= \max \left( 1, KL/ \left( R_{2}S_{2} \right) \right) c_{2}''$. We will set
our final value of $c_{2}$ to be $c_{2}'''$.

\subsubsection{The second condition in \eqref{eq:L-1}}

Recall the second condition in \eqref{eq:L-1}:
\[
\card \left\{ rb_{2}+sb_{1}: 0 \leq r< R_{2}, 0 \leq s<S_{2} \right\} > (K-1)L.
\]

Recall that here we have $b_{1}=1$ and $b_{2}=p$.
Suppose that $\left( r_{1}, s_{1} \right)$ and $\left( r_{2}, s_{2} \right)$
are distinct pairs of integers with $r_{1}b_{2}+s_{1}b_{1}=r_{2}b_{2}+s_{2}b_{1}$,
i.e., $\left( r_{1}-r_{2} \right)p=\left( s_{2}-s_{1} \right)$.
Then $p | \left( s_{2}-s_{1} \right)$. If $s_{1}=s_{2}$, then $r_{1}=r_{2}$, so
the pairs are not distinct. Hence if $p>S_{2}$, then the elements in the set on
the left-hand side of the second condition in \eqref{eq:L-1} are distinct, as we
assumed above.

\subsubsection{Calculations}

We proceed as in \cite{MV}, writing programs in $\mathtt{PARI/GP}$ based on those in \cite{MV},
but using Lemma~\ref{lem:Laurent-Thm1}. Whereas three iterations typically
sufficed with linear forms in three logs, we found here that five iterations were
often required. For each value of $k$, we search for values of $L$, $m$, $\varrho$
and $\mu$ that lead to values of $K$, $R_{1}$, $R_{2}$, $S_{1}$ and $S_{2}$
satisfying the conditions in Lemma~\ref{lem:Laurent-Thm1} and giving an upper
bound for $p$ that is as small as possible. In Appendix~\ref{sect:A-2}, we provide
tables with the values of $L$, $m$, $\varrho$ and $\mu$ used, and the upper bound
for $p$ that they provide.

In this way, we were able to improve the upper bounds on $p$ as follows:\\
-- for $k=4$, from $p \leq 251$ to $p \leq 191$;\\
-- for $k=5$, from $p \leq 283$ to $p \leq 211$;\\
-- for $k=6$, from $p \leq 307$ to $p \leq 229$;\\
-- for $k=7$, from $p \leq 317$ to $p \leq 241$.

The functions \verb!search_it1()! and
\verb!search_it2()! that we wrote in \verb!eg-gokhan.gp!, as well as functions in
related $\mathtt{PARI/GP}$ files, were used for these calculations. 

\section{Data for using Theorem~1 of \cite{Laurent} for \eqref{eq:Thue1}}
\label{sect:A-2}

Data for using Theorem~1 of \cite{Laurent} for \eqref{eq:Thue1} for $s = 12, 14, 16, 18$
(Tables~\ref{table:thm1-k4}, \ref{table:thm1-k5}, \ref{table:thm1-k6} and \ref{table:thm1-k7}).

\begin{table}[ht!]
\centering
\begin{tabular}{|c|c|c|c|c|c|c|} \hline
iteration & \text{initial upper bound for $p$} & $L$ & $m$     & $\varrho$ &  $\mu$  & \text{new upper bound for $p$} \\ \hline
$1$       & $256$                              & $5$ & $0.075$  &    $21.0$   & $0.475$ &  $209$ \\ \hline
$2$       & $209$                              & $5$ & $0.070$  &   $18.000$  & $0.575$ &  $196$ \\ \hline
$3$       & $196$                              & $5$ & $0.085$  &   $17.000$  & $0.500$ &  $193$ \\ \hline
$4$       & $193$                              & $5$ & $0.0695$ &   $17.975$  & $0.566$ &  $191$ \\ \hline
\end{tabular}
\caption{Data for $k=4$}
\label{table:thm1-k4}
\end{table}

\begin{table}[ht!]
\centering
\begin{tabular}{|c|c|c|c|c|c|c|} \hline
iteration & \text{initial upper bound for $p$} & $L$ & $m$     & $\varrho$ &  $\mu$  & \text{new upper bound for $p$} \\ \hline
$1$       & $284$                              & $5$ & $0.075$ &   $18.500$  & $0.55$  &  $231$ \\ \hline
$2$       & $231$                              & $5$ & $0.070$ &   $18.000$  & $0.575$ &  $218$ \\ \hline
$3$       & $218$                              & $5$ & $0.060$ &   $21.000$  & $0.550$ &  $216$ \\ \hline
$4$       & $216$                              & $5$ & $0.075$ &   $18.000$  & $0.525$ &  $214$ \\ \hline
$5$       & $214$                              & $5$ & $0.075$ &   $18.000$  & $0.524$ &  $213$ \\ \hline
\end{tabular}
\caption{Data for $k=5$}
\label{table:thm1-k5}
\end{table}

\begin{table}[ht!]
\centering
\begin{tabular}{|c|c|c|c|c|c|c|} \hline
iteration & \text{initial upper bound for $p$} & $L$ & $m$     & $\varrho$ &  $\mu$  & \text{new upper bound for $p$} \\ \hline
$1$       & $307$                              & $5$ & $0.080$ &   $19.000$  &  $0.5$  &  $252$ \\ \hline
$2$       & $252$                              & $5$ & $0.085$ &   $18.000$  & $0.475$ &  $238$ \\ \hline
$3$       & $238$                              & $5$ & $0.080$ &   $17.500$  & $0.550$ &  $234$ \\ \hline
$4$       & $234$                              & $5$ & $0.070$ &   $18.850$  & $0.568$ &  $232$ \\ \hline
$5$       & $232$                              & $5$ & $0.071$ &   $18.025$  & $0.552$ &  $231$ \\ \hline
\end{tabular}
\caption{Data for $k=6$}
\label{table:thm1-k6}
\end{table}

\begin{table}[ht!]
\centering
\begin{tabular}{|c|c|c|c|c|c|c|} \hline
iteration & \text{initial upper bound for $p$} & $L$ & $m$     & $\varrho$ &  $\mu$  & \text{new upper bound for $p$} \\ \hline
$1$       & $327$                              & $5$ & $0.075$ &   $20.0$  &  $0.5$  &  $268$ \\ \hline
$2$       & $268$                              & $5$ & $0.070$ &   $18.0$  & $0.575$ &  $254$ \\ \hline
$3$       & $254$                              & $5$ & $0.085$ &   $17.0$  &  $0.5$  &  $250$ \\ \hline
$4$       & $250$                              & $5$ & $0.075$ &   $18.0$  & $0.525$ &  $248$ \\ \hline
$5$       & $248$                              & $5$ & $0.070$ &   $19.0$  & $0.526$ &  $247$ \\ \hline
\end{tabular}
\caption{Data for $k=7$}
\label{table:thm1-k7}
\end{table}

Andrzej D\k{a}browski, Institute of Mathematics, University of Szczecin, Wielkopolska 15, 
70-451 Szczecin, Poland,  

 \begin{tabular}{rl}
E-mail: & dabrowskiandrzej7@gmail.com\\ 
& andrzej.dabrowski@usz.edu.pl 
\end{tabular}

\bigskip 

Salah Eddine Rihane, National Higher School of Mathematics, Scientific and Technology Hub of Sidi Abdellah, P.O. Box 75, Algiers 16093, Algeria,

E-mail:  salahrihane@hotmail.fr 

\bigskip 

G\"okhan Soydan, Department of Mathematics, Bursa Uluda\u{g} University, 16059 Bursa, T\"{u}rkiye,  

E-mail:  gsoydan@uludag.edu.tr 

\bigskip 

Paul Voutier, London, UK, 

E-mail: Paul.Voutier@gmail.com 

\begin{thebibliography}{999}
\bibitem{BM}
B. Bartolom\'{e}, P. Mih\u{a}ilescu,
{\it On equation $X^{n} - 1 = BZ^{n}$},
Int. J. Number Theory \textbf{13}(3) (2017), 549--570.  

\bibitem{Ben}
M.A. Bennett,
{\it Rational approximation to algebraic numbers of small height: the Diophantine equation $|ax^{n} + by^{n}| = 1$},
J. reine angew. Math. \textbf{535} (2001), 1--49.

\bibitem{BGMP}
M.A. Bennett, K. Gy\H{o}ry, M. Mignotte and \'{A}. Pint\'{e}r,
{\it Binomial Thue equations and polynomial powers},
Compos. Math. \textbf{142} (2006), 1103--1121.

\bibitem{BS}
M.A. Bennett and C.M. Skinner,
{\it Ternary diophantine equations via Galois representations and modular forms},
Canad. J. Math. \textbf{56}(1) (2004), 23--54.

\bibitem{BVY}
M.A. Bennett, V. Vatsal and S. Yazdani,
{\it Ternary Diophantine equations of signature $(p,p,3)$},
Compos. Math. \textbf{140} (2004), 1399--1416.

\bibitem{Magma}
W. Bosma, J. Cannon and C. Playoust,
{\it The Magma Algebra System I. The user language},
J. Symbolic Comput. \textbf{24} (1997), no. 3-4, 235--265.

\bibitem{Browkin}
J. Browkin,
{\it A weak effective $abc$-conjecture},
Funct. Approx. \textbf{39} (2008), 103--111.

\bibitem{BMS}
Y. Bugeaud, M. Mignotte and S. Siksek,
{\it A multi-Frey approach to some multi-parameter families of Diophantine equations},
Canad. J. Math. \textbf{60}(3) (2008), 491--519.

\bibitem{Cohen}
H. Cohen,
{\it Number Theory Volume I: Tools and Diophantine Equations},
Springer, Berlin, 2007.



\bibitem{DD}
E. Deza and M. M. Deza,
{\it Figurate Numbers},
World Scientific Publishing Co. Pte. Ltd., Hackensack, NJ, 2012.

\bibitem{D}
L.-E. Dickson,
{\it History of the Theory of Numbers. Vol. II: Diophantine Analysis},
Chelsea Publishing Co., New York, 1966.

\bibitem{Evertse}
J.-H. Evertse,
{\it On the equation $ax^{n} - by^{n} = c$},
Compos. Math. \textbf{47}(3) (1982), 289--315.

\bibitem{GDP}
K. Gy\H{o}ry, A. Dujella and \'{A}. Pint\'{e}r,
{\it On the power values of pyramidal numbers,I},
Acta Arith. \textbf{155(2)} (2012), 217--226.

\bibitem{GP}
K. Gy\H{o}ry and \'{A}. Pint\'{e}r,
{\it On the resolution of equations $Ax^{n} - By^{n} = C$ in integers $x, y$ and $n\geq 3$, I},
Publ. Math. Debrecen \textbf{70} (2007), 483--501.

\bibitem{HPTV}
L. Hajdu, \'{A}. Pint\'{e}r, S. Tengely and N. Varga,
{\it Equal values of figurate numbers},
Journal of Number Theory \textbf{137} (2014), 130--141.

\bibitem{HPTV-2}
B. He, \'{A}. Pint\'{e}r, A. Togb\'{e} and N. Varga,
{\it A Generalization of a problem of Mordell},
Glasnik Matemati\v{c}ki \textbf{50(70)} (2015), 35--41.

\bibitem{IR}
K. Ireland and M. Rosen,
{\it A Classical Introduction to Modern Number Theory},
2nd edition, Springer, 1990.

\bibitem{IK}
W. Ivorra and A. Kraus,
{\it Quelques r\'{e}sultats sur les \'{e}quations $ax^{p}+by^{p}=cz^{2}$},
Canad. J. Math. \textbf{58} (2006), 115--153.

\bibitem{KT}
M. Kaneko and K. Tachibana,
{\it When is a polygonal pyramid number again polygonal?},
Rocky Mountain J. Math. \textbf{32} (2002), 149--165.

\bibitem{KPP}
D. Kim, Y. K. Park and \'{A}. Pint\'{e}r,
{\it A Diophantine problem concerning polygonal numbers},
Bull. Aust. Math. Soc. \textbf{88} (2013), 345--350.

\bibitem{KR}
T. Kov\'{a}cs and Z. R\'{a}bai,
{\it Equal values of pyramidal numbers},
Indagationes Mathematicae \textbf{29(5)} (2018), 1157--1166.

\bibitem{Laurent}
M. Laurent,
{\it Linear forms in two logarithms and interpolation determinants II},
Acta Arith. \textbf{133}(4) (2008), 325--348.

\bibitem{M-R}
P. Michaud-Rodgers,
{\it A unique perfect power decagonal number},
Bull. Aust. Math. Soc. \textbf{105} (2022), 212--216.

\bibitem{Mi1}
M. Mignotte,
{\it A note on the equation $ax^{n} - by^{n} = c$},
Acta Arith. \textbf{75} (1996), 287--295.

\bibitem{MV}
M. Mignotte and P. Voutier (with an appendix by Michel Laurent),
{\it A kit for linear forms in three logarithms},
Math. Comp. \textbf{93} (2024), 1903--1951.

\bibitem{Pari}
The PARI Group,
$\mathtt{PARI/GP}$ version 2.14.0, Univ. Bordeaux, 2021,
\url{http://pari.math.u-bordeaux.fr/}.


\bibitem{TdW}
N. Tzanakis and B. M. M. de Weger,
{\it On the practical solution of the Thue equation},
J. Number Theory \textbf{31} (1989), 99--132.

\bibitem{Wald}
M. Waldschmidt,
{\it Diophantine Approximation on Linear Algebraic Groups},
Springer, Berlin, 2000.

\bibitem{Ya}
T. Yamada,
{\it A note on Laurent's paper on linear forms in two logarithms: the argument of an algebraic power},
Acta Arith. \textbf{221} (2025), 153--163.
\end{thebibliography}
\end{document}